\documentclass{article}

\usepackage{amsmath}
\usepackage{amssymb}
\usepackage{graphicx}
\usepackage{bm}
\usepackage{mathdots}
\usepackage{booktabs}

\usepackage[ruled,linesnumbered]{algorithm2e}
\usepackage[a4paper,left=2.8cm,right=2.8cm,top=2.5cm,bottom=2.5cm]{geometry}
\usepackage{fancyhdr}
\usepackage[framemethod=tikz]{mdframed}
\newmdtheoremenv[%
backgroundcolor=green!10,%
outerlinecolor=black,%
leftmargin=0,%
rightmargin=0,
innertopmargin =3pt,%
innerleftmargin = 5pt,
innerrightmargin = 5pt,
splittopskip = \topskip,%
skipabove = \baselineskip,%
skipbelow = \baselineskip,%
roundcorner=5, ntheorem]
{theorem}{Theorem}[section]
\newtheorem{corollary}{Corollary}[section]
\newtheorem{proposition}{Proposition}[section]

\newtheorem{definition}{Definition}[section]
\makeatletter % `@' now normal "letter"
\@addtoreset{equation}{section}
\makeatother  % `@' is restored as "non-letter"

\newtheorem{remark}{Remark}[section]

\newenvironment{proof}{{\noindent\it Proof.}\quad}{\hfill $\square$\\}

\usepackage{lineno}

\begin{document}

% \title{Minimization of the springback penalty and its application to signal reconstruction \\
% Minimization of the springback penalty for (robust) signal reconstruction (from limited data)\\
% Robust signal reconstruction via the springback penalty}

\title{The springback penalty for robust signal recovery}

\author{Congpei An\footnotemark[1]
       \quad\quad Hao-Ning Wu\footnotemark[2]
       \quad\quad  Xiaoming Yuan\footnotemark[2]}

\renewcommand{\thefootnote}{\fnsymbol{footnote}}
\footnotetext[1]{School of Economic Mathematics, Southwestern University of Finance and Economics, Chengdu, China (ancp@swufe.edu.cn).}
\footnotetext[2]{Department of Mathematics, The University of Hong Kong, Hong Kong, China (\{hnwu,xmyuan\}@hku.hk).}

% \date{}
%\shorttitle{SHORT TITLE}
%\shortauthor{F.~FIRSTA AND S.~SECONDA}
% \usepackage{keywords}

\maketitle

% \linenumbers

\begin{abstract}
We propose a new penalty, the springback penalty, for constructing models to recover an unknown signal from incomplete and inaccurate measurements. Mathematically, the springback penalty is a weakly convex function. It bears various theoretical and computational advantages of both the benchmark convex $\ell_1$ penalty and many of its non-convex surrogates that have been well studied in the literature. We establish the exact and stable recovery theory for the recovery model using the springback penalty for both sparse and nearly sparse signals, respectively, and derive an easily implementable difference-of-convex algorithm. In particular, we show its theoretical superiority to some existing models with a sharper recovery bound for some scenarios where the level of measurement noise is large or the amount of measurements is limited. We also demonstrate its numerical robustness regardless of the varying coherence of the sensing matrix. The springback penalty is particularly favorable for the scenario where the incomplete and inaccurate measurements are collected by coherence-hidden or -static sensing hardware due to its theoretical guarantee of recovery with severe measurements, computational tractability, and numerical robustness for ill-conditioned sensing matrices.
\end{abstract}

\textbf{Keywords.} {signal recovery, compressed sensing, penalty, weakly convex, difference-of-convex algorithm}\\

% \textbf{AMS subject classifications.}   94A12, 65K10, 90C26

\section{Introduction}
Signal recovery aims at recovering an unknown signal from its measurements, which are often incomplete and inaccurate due to technical, economical, or physical restrictions. Mathematically, a signal recovery problem can be expressed as estimating an unknown $\bar{x}\in\mathbb{R}^n$ from an underdetermined linear system
\begin{equation}\label{equ:linearsystem}
b = A\bar{x}+e,
\end{equation}
where $A\in\mathbb{R}^{m\times n}$ is a full row-rank sensing matrix such as a projection or transformation matrix (see, e.g., \cite{bruckstein2009sparse,candes2005error,candes2005decoding}), $b\in\mathbb{R}^m\setminus\{0\}$ is a vector of measurements, $e\in\mathbb{R}^m$ is some unknown but bounded noise perturbation in $\mathcal{B}(\tau):=\{e\in\mathbb{R}^m:\|e\|_2\leq\tau\}$, and the number $m$ of measurements is considerably smaller than the size $n$ of the signal $\bar{x}$. The set $\mathcal{B}(\tau)$ encodes both the cases of noise-free ($\tau=0$) and noisy ($\tau>0$) measurements.

Physically, a signal of interest, or its coefficients under certain transformation, is often sparse (see, e.g., \cite{bruckstein2009sparse}). Hence, it is natural to seek a sparse solution to the underdetermined linear system \eqref{equ:linearsystem}, though it has infinitely many solutions. We say that $x\in\mathbb{R}^n$ is \emph{$s$-sparse} if $\|x\|_0\leq s$, where $\|x\|_0$ counts the number of nonzero entries of $x$. To find the sparsest solution to \eqref{equ:linearsystem}, one may consider solving the following minimization problem:
\begin{equation}\label{equ:l0}
\min_{x\in\mathbb{R}^n}~\|x\|_0\quad\text{s.t.}\quad Ax-b\in\mathcal{B}(\tau),
\end{equation}
in which $\|x\|_0$ serves as a penalty term of the sparsity, and it is referred to as the $\ell_0$ penalty for convenience. Due to the discrete and discontinuous nature of the $\ell_0$ penalty, the model \eqref{equ:l0} is NP-hard \cite{bruckstein2009sparse}. This means the model \eqref{equ:l0} is computationally intractable, and this difficulty has inspired many alternatives to the $\ell_0$ penalty in the literature. A fundamental proxy of the model \eqref{equ:l0} is the basis pursuit (BP) problem proposed in  \cite{chen2001atomic}:
\begin{equation}\label{equ:l1}
\min_{x\in\mathbb{R}^n}~\|x\|_1\quad\text{s.t.}\quad Ax-b\in\mathcal{B}(\tau).
\end{equation}
In this convex model, $\|x\|_1:=\sum_{i=1}^n|x_i|$ and it is called the $\ell_1$ penalty hereafter. Recall that $\|x\|_1$ is the convex envelope of $\|x\|_0$ (see, e.g., \cite{MR0274683}), and it induces sparsity most efficiently among all convex penalties (see \cite{bruckstein2009sparse}). The BP problem \eqref{equ:l1} has been intensively studied in voluminous papers since the seminal works \cite{candes2006stable,candes2005error,donoho2006compressed}, in which various conditions have been comprehensively explored for the exact recovery via the convex model \eqref{equ:l1}.

The BP problem \eqref{equ:l1} is fundamental for signal recovery, but its solution may be over-penalized because the $\ell_1$ penalty tends to underestimate high-amplitude components of the solution, as analyzed in \cite{fan2001variable}. Hence, it is reasonable to consider non-convex alternatives to the $\ell_1$ penalty and upgrade the model \eqref{equ:l1} to achieve a more accurate recovery. In the literature, some non-convex penalties have been well studied, such as the smoothly clipped absolute deviation (SCAD) \cite{fan2001variable}, the capped $\ell_1$ penalty \cite{MR2629825}, the transformed $\ell_1$ penalty \cite{lv2009unified,zhang2018minimization}, and the $\ell_p$ penalty with $0<p<1$ \cite{chartrand2007exact,chartrand2008restricted,MR2765490}. Besides, one particular penalty is the minimax concave penalty (MCP) proposed in \cite{zhang2010nearly}, and it has been widely shown to be effective in reducing the bias from the $\ell_1$ penalty \cite{zhang2010nearly}. Moreover, the so-called $\ell_{1-2}$ penalty has been studied in the literature, e.g. \cite{esser2013method,yan2017sparse,yin2015minimization}, to mention a few. Some of these penalties will be summarized in Section \ref{sec:sb}. In a nutshell, convex penalties are more tractable in the senses of theoretical analysis and numerical computation, while they are less effective for achieving the desired sparsity (i.e., the approximation to the $\ell_0$ penalty is less accurate). Non-convex penalties are generally the opposite.

Considering the pros and cons of various penalties, we are motivated to find a weakly convex penalty that can keep some favorable features from both the $\ell_1$ penalty and its non-convex alternatives, and the resulting model for signal recovery is preferable in the senses of both theoretical analysis and numerical computation.
More precisely, we propose the \emph{springback} penalty
\begin{equation}\label{equ:spb}
\mathcal{R}^{\text{SPB}}_{\alpha}(x):=\|x\|_1-\frac{\alpha}{2}\|x\|_2^2,
\end{equation}
where $\alpha>0$ is a model parameter, and it should be chosen meticulously. We will show later that a larger $\alpha$ implies a tighter stable recovery bound. On the other hand, a too large $\alpha$ may lead to negative values of $\mathcal{R}^{\text{SPB}}_{\alpha}(x)$. Thus, a reasonable upper bound on $\alpha$ should be considered to ensure the well-definedness of the springback penalty \eqref{equ:spb}. In the following, we will see that if the matrix $A$ is well-conditioned (e.g., when $A$ is drawn from a Gaussian matrix ensemble), then the requirement on $\alpha$ is quite loose; while if $A$ is ill-conditioned (e.g., $A$ is drawn from an oversampled partial DCT matrix ensemble), then generally the upper bound on $\alpha $ should be better discerned for the sake of designing an algorithm with theoretically provable convergence. We refer to Theorem \ref{thm:sparse}, Theorem \ref{thm:nonsparse}, Section \ref{sec:convergence}, and Section \ref{sec:alpha} for more detailed discussions on the determination of $\alpha$ for the springback penalty \eqref{equ:spb} theoretically and numerically. With the springback penalty \eqref{equ:spb}, we propose the following model for signal recovery:
\begin{equation}\label{equ:CSproblem}
\min_{x\in\mathbb{R}^n}~\mathcal{R}^{\text{SPB}}_{\alpha}(x)\quad\text{s.t.}\quad Ax-b\in\mathcal{B}(\tau).
\end{equation}
Mathematically, the springback penalty \eqref{equ:spb} is a weakly convex function, and thus the \emph{springback-penalized model} \eqref{equ:CSproblem} can be intuitively regarded as an ``average" of the convex BP model \eqref{equ:l1} and the mentioned non-convex surrogates. Recall that a function $f:\mathbb{R}^n\rightarrow\mathbb{R}$ is \emph{$\alpha$-weakly convex} if $x\mapsto f(x)+\frac{\alpha}{2}\|x\|_2^2$ is convex. One advantage of the model \eqref{equ:CSproblem} is that various results developed in the literature on weakly convex optimization problems (e.g., \cite{guo2017convergence,mollenhoff2015primal}) can be used for both theoretical analysis and algorithmic design. Indeed, the weak convexity of the springback penalty \eqref{equ:spb} enables us to derive sharper recovery results with fewer measurements and to design some efficient algorithms easily.

The rest of this paper is organized as follows. In the next section, we summarize some preliminaries for further analysis. In Sections \ref{sec:signal} and \ref{sec:signal2}, we establish the exact and stable recovery theory of the springback-penalized model \eqref{equ:CSproblem} for sparse and nearly sparse signals, respectively. We also theoretically compare the springback penalty \eqref{equ:spb} with some other penalties in these two sections. In Section \ref{sec:computing}, we design a difference-of-convex algorithm (DCA) for the springback-penalized model \eqref{equ:CSproblem} and study its convergence. Some numerical results are reported in Section \ref{sec:numerical} to verify our theoretical assertions, and some conclusions are drawn in Section \ref{sec:conclusions}.

\section{Preliminaries}\label{sec:sb}
In this section, we summarize some preliminaries that will be used for further analysis.

\subsection{Notations}
For any $x,y\in\mathbb{R}^n$, let $\left<x,y\right>=x^{\text{T}}y$ be their inner product, and let $\text{supp}(x):=\{1\leq i\leq n:x_i\neq 0\}$ be the support of $x$. Let $I$ be an identity matrix whose dimension is clear in accordance with the context. Let $\Lambda\subseteq\{1,2,\ldots,n\}$ (or $\Lambda$ with some super/subscripts) be an index set, and $|\Lambda|$ the cardinality of $\Lambda$. For $x\in\mathbb{R}^n$ and $A\in\mathbb{R}^{m\times n}$, let $x_{\Lambda}\in\mathbb{R}^{n}$ be the vector with the same entries as $x$ on indices $\Lambda$ and zero entries on indices $\Lambda^c$, and let $A_{\Lambda}\in\mathbb{R}^{m\times |\Lambda|}$ be the submatrix of $A$ with column indices $\Lambda$. For $x\in\mathbb{R}$, $\text{sgn}(x)$ is the sign function of $x$. For a convex function $f$, $\partial(f(x))$ denotes the subdifferential of $f$ at $x$.

\subsection{A glance at various penalties}
In the literature, there are a variety of convex and non-convex penalties. Below we list six of the most important ones, with $x\in\mathbb{R}^n$.

\begin{itemize}
       \item[$\diamond$] The $\ell_1$ penalty \cite{bruckstein2009sparse,chen2001atomic}:
       \begin{equation*}
       \mathcal{R}^{\ell_1}(x):=\|x\|_1=\sum_{i=1}^n|x_i|.
       \end{equation*}
       \item[$\diamond$] The elastic net penalty \cite{MR2137327}:
       \begin{equation*}
       \mathcal{R}^{\text{EL}}(x):=\|x\|_1+\frac{\alpha}{2}\|x\|_2^2=\sum_{i=1}^n|x_i|+\frac{\alpha}{2}\sum_{i=1}^n|x_i|^2.
       \end{equation*}
       \item[$\diamond$] The $\ell_p$ penalty with parameter $0<p<1$ \cite{chartrand2007exact,chartrand2008restricted}:
       \begin{equation*}
       \mathcal{R}^{\ell_p}(x):=\|x\|_p^p=\sum_{i=1}^n|x_i|^p.
       \end{equation*}
       \item[$\diamond$] The transformed $\ell_1$ (TL1) with parameter $\beta>0$ \cite{lv2009unified,zhang2018minimization}:
       \begin{equation*}
       \mathcal{R}^{\text{TL1}}_{\beta}(x):=\sum_{i=1}^n\frac{(\beta+1)|x_i|}{\beta+|x_i|}.
       \end{equation*}
       \item[$\diamond$] The minimax concave penalty (MCP) with parameter $\mu>0$ \cite{zhang2010nearly}:
       \begin{equation}\label{equ:MCP}
       \mathcal{R}^{\text{MCP}}_{\mu}(x):=\sum_{i=1}^n\phi^{\text{MCP}}_{\mu}(x_i),
       \end{equation}
              where
              \begin{equation*}
              \phi^{\text{MCP}}_{\mu}(x_i)=\begin{cases}
              |x_i|-x_i^2/(2\mu),&|x_i|\leq\mu,\\
              \mu/2,&|x_i|\geq\mu.
              \end{cases}
              \end{equation*}
       \item[$\diamond$] The $\ell_{1-2}$ penalty \cite{esser2013method,yin2015minimization}:
              \begin{equation*}
              \mathcal{R}^{\ell_{1-2}}(x):=\|x\|_1-\|x\|_2=\sum_{i=1}^n|x_i|-\sqrt{\sum_{i=1}^n|x_i|^2}.
              \end{equation*}
\end{itemize}
Note that the $\ell_1$ penalty is convex, the elastic net penalty is strongly convex, and the others are non-convex.

\subsection{Relationship among various penalties}\label{sec:relation}
For any nonzero vector $x\in \mathbb{R}^n$ and $\alpha>0$, the springback penalty $\mathcal{R}^{\text{SPB}}_{\alpha}(x)\rightarrow\mathcal{R}^{\ell_1}(x)$ as $\alpha\rightarrow 0$. Besides, $\mathcal{R}^{\text{SPB}}_{\alpha}(x)$ is reduced to the MCP in \cite{zhang2010nearly} within the $\ell_{\infty}$-ball $\{x\in\mathbb{R}^n:\|x\|_{\infty}\leq\mu\}$ if $\alpha = 1/\mu$. The springback penalty appears to be a resemblance to the $\ell_{1-2}$ penalty, but their difference is many-sided. For instance, the gradient of $\|x\|_2$ is not defined at the origin.

Figure \ref{fig:penalty} displays some scalar (one-dimensional) penalties, including the $\ell_1$ penalty, the $\ell_{0.5}$ penalty, the transformed $\ell_1$ penalty with $\beta=1$, the MCP with $\mu = 0.75$, and the springback penalty with $\alpha = 1/\mu$ and $\alpha=0.15$. The $\ell_{1-2}$ penalty is not plotted, as it is none other than zero in the one-dimensional case. To give a better visual comparison, we scale them to attain the point $(1,1)$. It is shown in Figure \ref{fig:penalty} that the springback penalty is close to the $\ell_1$ penalty when $\alpha= 0.15$. The springback penalty with $\alpha = 1/\mu$ coincides with the MCP for $|x|\leq \mu$ if we do not scale them. The behavior of the springback penalty for $|x|>\mu$ attracts our interest because it turns around and heads towards the $x$-axis. According to Figure \ref{fig:penalty}, this behavior is clearer in terms of the thresholding operator corresponding to the proximal mapping of the springback penalty, whose mathematical descriptions are given in Section \ref{sec:proximal}.

\begin{figure}[htbp]
  \centering
  \includegraphics[width=\textwidth]{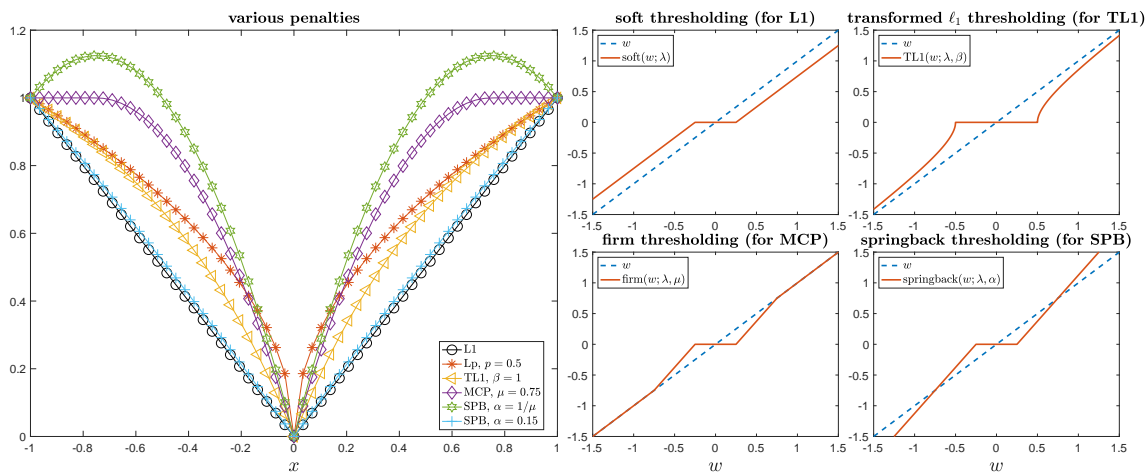}
  \caption{Scalar penalties and corresponding thresholding operators (for representing proximal mappings with $\lambda=0.25$): the $\ell_1$ penalty and the soft thresholding operator; the $\ell_p$ penalty, whose proximal mapping has no closed-form expressions (hence no thresholding operator plotted); the transformed $\ell_1$ penalty with $\beta=1$, whose proximal mapping can be expressed explicitly by a thresholding operator given in \cite{MR3620567}; the MCP with $\mu=0.75$ and the firm thresholding operator; and two springback penalties with $\alpha=1/\mu$ and $\alpha=0.15$, and the springback thresholding operator.}\label{fig:penalty}
\end{figure}

As mentioned, the proposed springback penalty \eqref{equ:spb} balances the approximation quality of the $\ell_0$ penalty and the tractability in analysis and computation, and it is in between the convex and non-convex penalties. More specifically, it is in between the $\ell_1$ penalty and the MCP. For any $x\in\mathbb{R}^n$, we can always find a parameter $\mu$ for the MCP such that $\|x\|_{\infty}\leq\mu$ with a resulting penalty in the form of $\|x\|_1-\|x\|_2^2/(2\mu)$. This penalty inherits the approximation quality of the $\ell_0$ penalty from the MCP and the analytical and computational advantages of the $\ell_1$ penalty. Inasmuch as this penalty, we consider the more general penalty \eqref{equ:spb} in which $1/\mu$ is replaced by a more flexible parameter $\alpha>0$.

\subsection{Proximal mappings and thresholding operators}\label{sec:proximal}
For a function $\mathcal{R}:\mathbb{R}^n\rightarrow\mathbb{R}$, as defined in \cite{MR201952}, the proximal mapping of $\mathcal{R}$ is defined as
\begin{equation}\label{equ:prox}
\text{prox}_{\lambda}\left[\mathcal{R}\right](x):=\arg\min_{y\in\mathbb{R}^n}~\left\{\lambda\mathcal{R}(y)+\frac{1}{2}\|y-x\|^2_2\right\},
\end{equation}
where $\lambda>0$ is a regularization parameter. In \eqref{equ:prox}, we slightly abuse the notation ``=". This mapping takes a vector $x\in\mathbb{R}^n$ and maps it into a subset of $\mathbb{R}^n$, which might be empty, a singleton, or a set with multiple vectors; and the image of $y$ under this mapping is a singleton if the function $\mathcal{R}$ is proper closed and convex \cite{MR3719240}. For a given optimization model, if the proximal mapping of its objective function has a closed-form expression, then usually it is important and necessary to consider how to take advantage of this feature for algorithmic design.

When the proximal mapping of a penalty can be represented explicitly, the closed-form representation is often called a \emph{thresholding operator} or a shrinkage operator in the literature. For example, as analyzed in \cite{MR3620567}, with the \emph{soft thresholding operator}
\begin{equation*}
{\rm{soft}}(w;\lambda)={{\rm{sgn}}(w)}\max\{|w|-\lambda,0\},
\end{equation*}
which has been widely used in various areas such as compressed sensing and image processing, the proximal mapping \eqref{equ:prox} of the $\ell_1$ penalty can be expressed explicitly by
\begin{equation*}
\left[\text{prox}_{\lambda}\left[\mathcal{R}^{\ell_1}\right](x)\right]_i={\rm{soft}}(x_i;\lambda), \quad i=1,\ldots,n.
\end{equation*}

The proximal mapping of a non-convex penalty, in general, does not have a closed-form expression; such cases include the $\ell_{1-2}$ penalty and the $\ell_p$ penalty with $0<p<1$. However, there are some particular non-convex penalties whose proximal mappings can still be represented explicitly. For instance, the transformed $\ell_1$ penalty \cite{MR3620567} and the MCP \cite{zhang2010nearly}. In particular, with the following \emph{firm thresholding operator}
\begin{equation*}
{\rm{firm}}(w;\lambda,\mu)=\begin{cases}
0,& |w|\leq \lambda,\\
{\rm{sgn}}(w)\frac{\mu(|w|-\lambda)}{\mu-\lambda},& \lambda\leq |w|\leq \mu,\\
w, & |w|\geq\mu,
\end{cases}
\end{equation*}
which was first proposed in \cite{gao1997waveshrink}, it was further studied in \cite{zhang2010nearly} that the proximal mapping \eqref{equ:prox} of the MCP can be expressed explicitly by a firm thresholding operator for the case of orthonormal designs. More specifically, the proximal mapping \eqref{equ:prox} of the MCP is
\begin{equation*}
\left[\text{prox}_{\lambda}\left[\mathcal{R}^{\rm{MCP}}_{\mu}\right](x)\right]_i = \text{firm}(x_i;\lambda,\mu), \quad i=1,\ldots,n.
\end{equation*}

Below, we show that for the springback penalty \eqref{equ:spb} with a well chosen $\alpha$, its proximal mapping can also be expressed explicitly.

\begin{definition}
The \emph{springback thresholding operator} is defined as
\begin{equation}\label{equ:sbshrinkage}
{\rm{springback}}(w;\lambda,\alpha)=\begin{cases}
0,& |w|\leq \lambda,\\
{\rm{sgn}}(w)\frac{|w|-\lambda}{1-\lambda\alpha},&  |w|>\lambda.
\end{cases}
\end{equation}
\end{definition}
\begin{proposition}\label{prop:shrinkage}
If $1-\lambda\alpha>0$, then the proximal mapping of the springback penalty \eqref{equ:spb} can be represented explicitly as
\begin{equation*}
\left[{\rm{prox}}_{\lambda}\left[\mathcal{R}^{\rm{SPB}}_{\alpha}\right](x)\right]_i={\rm{springback}}(x_i;\lambda,\alpha),  \quad i=1,\ldots,n.
\end{equation*}
\end{proposition}
\begin{proof}
When $\mathcal{R}(x)=\mathcal{R}^{\ell_1}(x)$, it follows from \eqref{equ:prox} that, for any $z\in\mathbb{R}^n$ satisfying $0\in z-y+\lambda\partial\left(\|z\|_1\right)$, there holds $z_i={\rm{soft}}(y_i;\lambda)$, i.e., $z=\text{prox}_{\lambda}\left[\mathcal{R}^{\ell_1}\right](y)$. The assumption $1-\lambda\alpha>0$ ensures $\nabla^2\left(\frac12\|x-y\|_2^2-\frac{\lambda\alpha}{2}\|x\|_2^2\right)=(1-\lambda\alpha)I$ to be positive definite. Thus, the optimization problem occurred in \eqref{equ:prox} is convex. When $\mathcal{R}(x)=\mathcal{R}^{\rm{SPB}}_{\alpha}(x)$ in \eqref{equ:prox}, for any $z\in\mathbb{R}^n$ satisfying the condition $0\in z-y+\lambda \partial (\|z\|_1)-\lambda\alpha z$, which is equivalent to
\begin{equation}\label{equ:first}
 0 \in z-\frac{1}{1-\lambda\alpha}y+\frac{\lambda}{1-\lambda\alpha}\partial(\|z\|_1),
\end{equation}
we have $z = \text{prox}_{\lambda}\left[\mathcal{R}^{\rm{SPB}}_{\alpha}\right](y)$. It also follows from \eqref{equ:first} that
\begin{equation*}
z_i=\text{soft}\left(\frac{y_i}{1-\lambda\alpha};\frac{\lambda}{1-\lambda\alpha}\right)=\text{springback}(y_i;\lambda,\alpha).
\end{equation*}
Hence, the assertion is proved.
\end{proof}

Recall that the springback penalty \eqref{equ:spb} is a weakly convex function. Its thresholding operator defined in \eqref{equ:sbshrinkage} is also in between the soft and firm thresholding operators. As $\lim_{\mu\rightarrow\infty}{\rm{firm}}(w;\lambda,\mu)={\rm{soft}}(w;\lambda)$, a compromising $\mu$ could be large enough such that $|w|\leq\mu$ and it reaches a certain compromise between the soft and firm thresholding operators. In this case, we have a particular springback thresholding operator
\begin{equation*}
{\rm{springback}}(w;\lambda,1/\mu)=\begin{cases}
0,& |w|\leq \lambda,\\
{\rm{sgn}}(w)\frac{\mu(|w|-\lambda)}{\mu-\lambda},& |w|\geq\lambda.
\end{cases}
\end{equation*}
If $1/\mu$ is replaced by a more general $\alpha>0$, then the springback thresholding operator \eqref{equ:sbshrinkage} is recovered.

\subsection{Rationale of the name}
Springback is a concept in applied mechanics (see, e.g., \cite{todd1994manufacturing}).  Figure \ref{fig:penalty} gives more explanations for naming \eqref{equ:spb} \emph{springback}. With $\lambda=0.25$, Figure \ref{fig:penalty} displays the thresholding operators for $w\in[-1.5,1.5]$, including the soft thresholding operator, the transformed $\ell_1$ thresholding operator with $\beta = 1$, the firm thresholding operator with $\mu=0.75$, and the springback thresholding operator with $\alpha=1/\mu$. The transformed $\ell_1$ thresholding operator enforces $w$ with $|w|\leq \lambda(\beta+1)/\beta$ to be 0, and then its outputs approach to $w$ as $|w|$ increases. All the other thresholding operators enforce $w$ with $|w|\leq\lambda$ to be $0$. For $w\geq \lambda$, the soft thresholding operator subtracts $\lambda$ from $|w|$ and thus causes the $\ell_1$ penalty to underestimate high-amplitude components; the firm thresholding operator's outputs jump from 0 to $\mu$ until $|w|$ exceeds $\mu$, afterwards its output is $w$. For the springback thresholding operator, its outputs jump from 0 to $\mu$ until $|w|$ exceeds $\mu$, and afterwards its outputs still keep going along the previous jumping trajectory.

In applied mechanics, spring is related to the process of bending some materials. When the bending process is done, the residual stresses cause the material to \emph{spring back} towards its original shape, so the material must be \emph{over-bent} to achieve the proper bending angle. Note that the soft thresholding operator always underestimates high-amplitude components, and the components $\|x\|_1$ and $-\frac{\alpha}{2}\|x\|_2^2$ in the springback penalty are decoupled. If we deem the soft thresholding operator as a process of over-bending, which stems for the component $\|x\|_1$, then the output of the soft thresholding operator will be sprung back toward $w$, which is achieved separately in consideration with the component $-\frac{\alpha}{2}\|x\|_2^2$. Such a springback process occurs for both $\lambda\leq|w|\leq\mu$ and $|w|\geq\mu$. The springback behavior is more obvious for those $w$ with larger absolute values, and this coincides with the behavior of the springback penalty in Figure \ref{fig:penalty}. That is, once $|x|$ exceeds $\mu$, the penalty turns around and heads towards the $x$-axis. This process may also be explained as a compensation of the loss of $|w|$ with $|w|\leq\lambda$.

\section{Springback-penalized model for sparse signal recovery}\label{sec:signal}

In this section, we focus on the recovery of a sparse signal using the springback-penalized model \eqref{equ:CSproblem}. After reviewing some basic knowledge of compressed sensing, we identify some conditions for exact and robust recovery using the springback-penalized model \eqref{equ:CSproblem}, respectively.

\subsection{Compressed sensing basics}
In some seminal compressed sensing papers such as \cite{candes2006robust,donoho2006compressed}, recovery conditions have been established for the BP model \eqref{equ:l1}. These conditions rely on the restricted isometry property (RIP) of the \emph{sensing matrix} $A$, as proposed in \cite{candes2005decoding}.
\begin{definition}
 For an index set $T\subset\{1,2,\ldots,n\}$ and an integer $s$ with $|T|\leq s$, the $s$-\emph{restricted isometry constant} (RIC) of ${A}\in\mathbb{R}^{m\times n}$ is the smallest $\delta_s\in(0,1)$ such that
\begin{equation*}
(1-\delta_s)\|x\|_2^2\leq\|A_Tx\|_2^2\leq(1+\delta_s)\|x\|_2^2
\end{equation*}
for all subsets $T$ with $|T|\leq s$ and all $x\in\mathbb{R}^{|T|}$. The matrix ${A}$ is said to satisfy the $s$-\emph{restricted isometry property} (RIP) with $\delta_s$.
\end{definition}

Denoting by $x^{\text{opt}}$ the minimizer of the BP problem \eqref{equ:l1}, if $A$ satisfies $\delta_{3s}<3(1-\delta_{4s})-1$, then for an $s$-sparse $\bar{x}$, one has
\begin{equation}\label{equ:l1sparserecovery}
\|x^{\text{opt}}-\bar{x}\|_2\leq C_s\tau,
\end{equation}
where $C_s$ is a constant which may only depend on $\delta_{4s}$. We refer to \cite{candes2006stable,candes2005error} for more details. If the measurements are noise-free, i.e., $\tau=0$, then the error bound \eqref{equ:l1sparserecovery} implies \emph{exact recovery}. Exact recovery is guaranteed only in the idealized situation where $\bar{x}$ is $s$-sparse and the measurements are noise-free. If the measurements are perturbed by some noise, then the bound \eqref{equ:l1sparserecovery} is usually referred to as the \emph{robust recovery} result with respect to the measurement noise. In more realistic scenarios, we can only claim that $\bar{x}$ is close to an $s$-sparse vector, and the measurements may also be contaminated. In such cases, we can recover $\bar{x}$ with an error controlled by its distance to $s$-sparse vectors, and it was proved in \cite{candes2006stable} that
\begin{equation}\label{equ:l1nonsparserecovery}
\|x^{\text{opt}}-\bar{x}\|_2\leq C_{1,s}\tau+C_{2,s}\frac{\|\bar{x}-\bar{x}_{s}\|_1}{\sqrt{s}},
\end{equation}
where $\bar{x}_{s}$ is the truncated vector corresponding to the $s$ largest values of $\bar{x}$ (in absolute value), and $C_{1,s}$ and $C_{2,s}$ are two constants which may only depend on $\delta_{4s}$. The bound \eqref{equ:l1nonsparserecovery} is usually referred to as the \emph{stable recovery} results. Recovery conditions for other models with different penalties are usually not as extensive as the BP model \eqref{equ:l1}. Under the framework of the RIP or some generalized versions, recovery theory for the BP model \eqref{equ:l1} has been generalized to the $\ell_p$-penalized model in \cite{chartrand2007exact,foucart2009sparsest}. With the \emph{unique representation property} of $A$, stable recovery results for the MCP-penalized model were derived in \cite{woodworth2016compressed} and an upper bound for $\sum_{i=1}^n\phi^{\text{MCP}}_{\mu}(x^{\text{opt}}_i-\bar{x}_i)$, but not for $\|x^{\text{opt}}-\bar{x}\|_2$, was obtained. We recommend the monograph \cite{MR3100033} for a more comprehensive and detailed exhibition on compressed sensing.

\subsection{Recovery guarantee using the springback-penalized model}
Still denoting by $x^{\text{opt}}$ the minimizer of the springback-penalized model \eqref{equ:CSproblem}, we have the following exact and robust recovery results of the model \eqref{equ:CSproblem} for an $s$-sparse $\bar{x}$.

\begin{theorem}[recovery of sparse signals]\label{thm:sparse}
Let $\bar{x}\in\mathbb{R}^n$ be an unknown $s$-sparse vector to be recovered. For a given sensing matrix $A\in\mathbb{R}^{m\times n}$, let $b\in\mathbb{R}^m$ be a vector of measurements from $b=A\bar{x}+e$ with $\|e\|_2\leq\tau$, and let $\delta_{3s}$ and $\delta_{4s}$ be the $3s$- and $4s$-RIC's of $A$, respectively. Suppose $A$ satisfies $\delta_{3s}<3(1-\delta_{4s})-1$ and $\alpha$ satisfies
\begin{equation}\label{equ:alpha}
\alpha\leq\frac{\sqrt{1-\delta_{4s}}\sqrt{3s}-\sqrt{1+\delta_{3s}}\sqrt{s}}{(\sqrt{1-\delta_{4s}}+\sqrt{1+\delta_{3s}})\|x^{\rm{opt}}\|_{2}},
\end{equation}
then the minimizer $x^{\rm{opt}}$ of the problem \eqref{equ:CSproblem} satisfies $x^{\rm{opt}}=\bar{x}$ when $\tau=0$; and it satisfies
\begin{equation}\label{equ:stablerecovery}
\|x^{\rm{opt}}-\bar{x}\|_2\leq \frac{\sqrt{2}}{\sqrt{D_1}}\sqrt{\tau}
\end{equation}
when $\tau\geq0$,
where
\begin{equation}\label{equ:D1}
D_1=\frac{\alpha}{2}\frac{\sqrt{1-\delta_{4s}}+\sqrt{1+\delta_{3s}}}{\sqrt{3s}+\sqrt{s}}.
\end{equation}
\end{theorem}

\begin{proof}
Let $x^{\rm{opt}}=\bar{x}+v$, and $\Lambda_0$ be the support of $\bar{x}$. It is clear that $v_{\Lambda_0}=x^{\rm{opt}}_{\Lambda_0}-\bar{x}$ and $v_{\Lambda_0^c}=x^{\rm{opt}}_{\Lambda_0^c}$. On the one hand, we know that
\begin{equation*}
\|x^{\rm{opt}}\|_1-\frac{\alpha}{2}\|x^{\rm{opt}}\|_2^2\leq\|\bar{x}\|_1-\frac{\alpha}{2}\|\bar{x}\|_2^2.
\end{equation*}
On the other hand, it holds that
\begin{equation*}\begin{split}
\|x^{\rm{opt}}\|_1-\frac{\alpha}{2}\|x^{\rm{opt}}\|_2^2
=&\|\bar{x}+v_{\Lambda_0}\|_1+\|v_{\Lambda_0^c}\|_1-\frac{\alpha}{2}\|\bar{x}+v\|_2^2\\
\geq& \|\bar{x}\|_1-\|v_{\Lambda_0}\|_1+\|v_{\Lambda_0^c}\|_1-\frac{\alpha}{2}\left(\|\bar{x}\|_2^2+2\left<\bar{x},v\right>+\|v\|_2^2\right).
\end{split}\end{equation*}
Then, we have that
\begin{equation*}\|v_{\Lambda_0^c}\|_1
\leq\|v_{\Lambda_0}\|_1-\frac{\alpha}{2}\|v\|_2^2+\alpha\|v\|_2^2+\alpha\left<\bar{x},v\right>= \|v_{\Lambda_0}\|_1-\frac{\alpha}{2}\|v\|_2^2+\alpha\left<x^{\rm{opt}},v\right>.\end{equation*}

We continue by arranging the indices in $\Lambda_0^c$ in order of decreasing magnitudes (in absolute value) of $v_{\Lambda_0^c}$, and then dividing $\Lambda_0^c$ into subsets of size $3s$. Set $\Lambda_0^c=\Lambda_1\bigcup\Lambda_2\bigcup\cdots\bigcup\Lambda_{\ell}$, i.e., $\Lambda_1$ contains the indices of the $3s$ largest entries (in absolute value) of $v_{\Lambda_0^c}$, $\Lambda_2$ contains the indices of the next $3s$ largest entries (in absolute value) of $v_{\Lambda_0^c}$, and so on. The cardinal number of $\Lambda_{\ell}$ may be less than $3s$. Denoting $\Lambda_{01}=\Lambda_0\bigcup\Lambda_1$ and using the RIP of $A$, we have
\begin{equation*}
\|Av\|_2\geq\|A_{\Lambda_{01}}v_{\Lambda_{01}}\|_2-\left\|\sum_{i=2}^{\ell}A_{\Lambda_i}v_{\Lambda_i}\right\|_2\geq\sqrt{1-\delta_{4s}}\|v_{\Lambda_{01}}\|_2-\sqrt{1+\delta_{3s}}\sum_{i=2}^{\ell}\|v_{\Lambda_i}\|_2.
\end{equation*}
As the magnitude of every $v_t$ indexed by $t\in\Lambda_{i+1}$ is less than the average of magnitudes of $v_t$ indexed by $t\in\Lambda_i$, there holds
$|v_{t}|\leq\frac{\|v_{\Lambda_i}\|_1}{3s}$, where $t\in\Lambda_{i+1}$. Then, we have
\begin{equation*}
\|v_{\Lambda_{i+1}}\|_2^2\leq 3s\frac{\|v_{\Lambda_i}\|_1^2}{(3s)^2}=\frac{\|v_{\Lambda_i}\|_1^2}{3s}.
\end{equation*}
Together with $\|v_{\Lambda_0}\|_1\leq\sqrt{s}\|v_{\Lambda_0}\|_2\leq\sqrt{s}\|v_{\Lambda_{01}}\|_2$, we have
\begin{equation*}\begin{split}
\sum_{i=2}^{\ell}\|v_{\Lambda_i}\|_2\leq\sum_{i=1}^{\ell-1}\frac{\|v_{\Lambda_i}\|_1}{\sqrt{3s}}\leq\frac{1}{\sqrt{3s}}\|v_{\Lambda_0^c}\|_1\leq\frac{1}{\sqrt{3s}}\left(\sqrt{s}\|v_{\Lambda_{01}}\|_2-\frac{\alpha}{2}\|v\|_2^2+\alpha\left<x^{\rm{opt}},v\right>\right).
\end{split}\end{equation*}
Thus, it holds that
\begin{equation}\label{equ:Avcase2}
\|Av\|_2\geq\left(\sqrt{1-\delta_{4s}}-\frac{\sqrt{1+\delta_{3s}}\sqrt{s}}{\sqrt{3s}}\right)\|v_{\Lambda_{01}}\|_2+\frac{\alpha\sqrt{1+\delta_{3s}}}{2\sqrt{3s}} \|v\|_2^2-\frac{\alpha\sqrt{1+\delta_{3s}}}{\sqrt{3s}}\left<x^{\rm{opt}},v\right>.
\end{equation}
Note that
\begin{equation*}
\|v\|_2\leq\|v_{\Lambda_{01}}\|_2+\sum_{i=2}^{\ell}\|v_{\Lambda_i}\|_2
\leq\left(1+\frac{\sqrt{s}}{\sqrt{3s}}\right)\|v_{\Lambda_{01}}\|_2-\frac{\alpha}{2\sqrt{3s}} \|v\|_2^2+\frac{\alpha}{\sqrt{3s}}\left<x^{\rm{opt}},v\right>,
\end{equation*}
and it can be written as
\begin{equation*}
\|v_{\Lambda_{01}}\|_2\geq\frac{\sqrt{3s}}{\sqrt{3s}+\sqrt{s}}\left(\frac{\alpha}{2\sqrt{3s}}\|v\|_2^2+\|v\|_2-\frac{\alpha}{\sqrt{3s}}\left<x^{\rm{opt}},v\right>\right).
\end{equation*}
With the assumption $\delta_{3s}<3(1-\delta_{4s})-1$ on $A$, the coefficient of $\|v_{\Lambda_{01}}\|_2$ in \eqref{equ:Avcase2} is positive and thus we have
    \begin{equation}\label{equ:Av}\begin{split}
    \|Av\|_2\geq&\frac{\sqrt{1-\delta_{4s}}\sqrt{3s}-\sqrt{1+\delta_{3s}}\sqrt{s}}{\sqrt{3s}+\sqrt{s}}\left(\frac{\alpha}{2\sqrt{3s}}\|v\|_2^2+\|v\|_2-\frac{\alpha}{\sqrt{3s}}\left<x^{\rm{opt}},v\right>\right)\\
    &+\frac{\alpha\sqrt{1+\delta_{3s}}}{2\sqrt{3s}} \|v\|_2^2-\frac{\alpha\sqrt{1+\delta_{3s}}}{\sqrt{3s}}\left<x^{\rm{opt}},v\right>\\
    =&\frac{\alpha}{2}\left(\frac{\sqrt{1-\delta_{4s}}+\sqrt{1+\delta_{3s}}}{\sqrt{3s}+\sqrt{s}}  \right)\|v\|_2^2+\frac{\sqrt{1-\delta_{4s}}\sqrt{3s}-\sqrt{1+\delta_{3s}}\sqrt{s}}{\sqrt{3s}+\sqrt{s}}\|v\|_2\\
    &-\alpha\left(\frac{\sqrt{1-\delta_{4s}}+\sqrt{1+\delta_{3s}}}{\sqrt{3s}+\sqrt{s}}  \right)\left<x^{\rm{opt}},v\right>.
    \end{split}\end{equation}
    If $\left<x^{\rm{opt}},v\right>\leq0$, then $\|Av\|_2\geq D_1\|v\|_2^2$. If $\left<x^{\rm{opt}},v\right>>0$, then the condition \eqref{equ:alpha} on $\alpha$ guarantees
\begin{equation*}\begin{split}
&\frac{\sqrt{1-\delta_{4s}}\sqrt{3s}-\sqrt{1+\delta_{3s}}\sqrt{s}}{\sqrt{3s}+\sqrt{s}}\|v\|_2-\alpha\left(\frac{\sqrt{1-\delta_{4s}}+\sqrt{1+\delta_{3s}}}{\sqrt{3s}+\sqrt{s}}  \right)\left<x^{\rm{opt}},v\right>\\
\geq &\frac{\sqrt{1-\delta_{4s}}\sqrt{3s}-\sqrt{1+\delta_{3s}}\sqrt{s}}{\sqrt{3s}+\sqrt{s}}\left(\|v\|_2-\left<\frac{x^{\rm{opt}}}{\|x^{\rm{opt}}\|_{2}},v\right>\right)\geq 0,
\end{split}\end{equation*}
where we use the Cauchy--Schwarz inequality. Hence we also have $\|Av\|_2\geq D_1\|v\|_2^2$.

When $\tau=0$, the inequality $Av=A(x^{\rm{opt}}-\bar{x})=b-b=0$ renders $0=\|Av\|_2\geq D_1\|v\|_2^2$, which implies $\|v\|_2=0$. Thus $x^{\rm{opt}}=\bar{x}$. When $\tau>0$, the inequality
\begin{equation*}
\|Av\|_2=\|Ax^{\rm{opt}}-A\bar{x}\|_2\leq\|Ax^{\rm{opt}}-b\|_2+\|A\bar{x}-b\|_2\leq2\tau
\end{equation*}
leads to $2\tau\geq D_1\|v\|_2^2$, which implies \eqref{equ:stablerecovery}.
\end{proof}

In analysis of signal recovery models with various convex and non-convex penalties, such as the $\ell_1$ penalty \cite{candes2005error,chartrand2007exact} and the $\ell_{1-2}$ penalty \cite{yan2017sparse,yin2015minimization}, a linear lower bound for $\|A(x^{\rm{opt}}-\bar{x})\|_2$ is derived somehow. The proof of Theorem \ref{thm:sparse} mainly follows the idea of \cite{candes2005error}, but we derive a quadratic lower bound for the term $\|A(x^{\rm{opt}}-\bar{x})\|_2$.  Thus, it is worthy noting that our results cannot be reduced to the result of the BP model \eqref{equ:l1} as $\alpha\rightarrow0$. Indeed, the quadratic bound \eqref{equ:Avcase2} in our proof is reduced to a linear bound as $\alpha\rightarrow0$, which then leads to the same results as the BP model \eqref{equ:l1}. However, we handle our final quadratic bound by removing its linear and constant terms and hence the obtained result cannot be reduced to the result of the BP model \eqref{equ:l1} as $\alpha\rightarrow0$.

Besides, the condition \eqref{equ:alpha} on $\alpha$ is required for the springback-penalized model \eqref{equ:CSproblem}. It is impossible to choose an $\alpha$ satisfying \eqref{equ:alpha} unless we have \emph{a priori} estimation on $\|x^{\rm{opt}}\|_2$ before solving the problem \eqref{equ:CSproblem}. Thus, the condition \eqref{equ:alpha} then can be interpreted as a \emph{posterior verification} in the sense that it can be verified once $x^{\rm{opt}}$ is obtained by solving the problem \eqref{equ:CSproblem}.

\begin{remark}[Posterior verification]
In practice, we solve the springback-penalized model \eqref{equ:CSproblem} numerically and thus obtain an approximate solution, denoted by $x^*$, subject to a preset accuracy $\epsilon>0$. That is,  $\|x^{\rm{opt}}-x^*\|_2\leq\epsilon$. Then, the posterior verification \eqref{equ:alpha} is guaranteed if
\begin{equation*}
\alpha\leq\frac{\sqrt{1-\delta_{4s}}\sqrt{3s}-\sqrt{1+\delta_{3s}}\sqrt{s}}{(\sqrt{1-\delta_{4s}}+\sqrt{1+\delta_{3s}})(\|x^*\|_2+\epsilon)}.
\end{equation*}
\end{remark}

\subsection{On the exact and robust recovery}
In Theorem \ref{thm:sparse}, we establish conditions for exact and robust recovery using the springback-penalized model \eqref{equ:CSproblem}. Table \ref{tab:exact} lists the exact recovery conditions for five other popular models in the literature. In particular, the springback-penalized model \eqref{equ:CSproblem} and the $\ell_1$-penalized model, i.e., the BP model \eqref{equ:l1}, have the same RIP condition. This condition is more stringent than that of the $\ell_p$-penalized model ($0<p<1$) but weaker than those of the transformed $\ell_1$- and $\ell_{1-2}$-penalized models. Beside the RIP condition, there is an additional assumption $a(s)>1$ for the $\ell_{1-2}$-penalized model, where $a(s)$ was first derived in \cite{yin2015minimization} and slightly improved in \cite{yan2017sparse} as
\begin{equation*}
a(s)=\left(\frac{3s-1}{\sqrt{3}s+\sqrt{4s-1}}\right)^2.
\end{equation*}
Note that $a(s)<3$ was shown in \cite{yan2017sparse,yin2015minimization} for both the cases.

\begin{table}[htbp]
  \centering
  \setlength{\abovecaptionskip}{0pt}
\setlength{\belowcaptionskip}{10pt}
  \caption{Exact recovery conditions recovery models with various penalties.}\label{tab:exact}
  \begin{tabular}{c|c}
  \hline
    Penalty  & RIP condition   \\ \hline
    $\ell_1$  \cite{candes2005error}  & $\delta_{3s}<3(1-\delta_{4s})-1$ \\ \hline
    $\ell_p$  ($0<p<1$) \cite{chartrand2007exact} & $\delta_{3s}<3^{(2-p)/p}(1-\delta_{4s})-1$ \\ \hline
    transformed $\ell_1$ \cite{zhang2018minimization} & $\delta_{3s}<\left(\frac{\beta}{\beta+1}\right)^23(1-\delta_{4s})-1$ \\ \hline
    $\ell_{1-2}$ \cite{yan2017sparse,yin2015minimization} & $\delta_{3s}<a(s)(1-\delta_{4s})-1$  \\ \hline
    springback  & $\delta_{3s}<3(1-\delta_{4s})-1$\\
    \hline
  \end{tabular}
\end{table}

We then discuss robust recovery results. If $\alpha\rightarrow0$, then the result \eqref{equ:stablerecovery} cannot provide any information as $\frac{\sqrt{2}}{\sqrt{D_1}}\rightarrow\infty$. However, for an appropriate $\alpha$, the bound \eqref{equ:stablerecovery} is informative and attractive. The robust recovery results of the $\ell_1$-, $\ell_p$-, transformed $\ell_1$- and $\ell_{1-2}$-penalized models were shown to be linear with respect to the level of noise $\tau$ \cite{candes2005error,chartrand2007exact,yan2017sparse,yin2015minimization,zhang2018minimization}, in the sense of
\begin{equation}\label{equ:linear}
\|x^{\rm{opt}}-\bar{x}\|_2\leq C_s \tau,
\end{equation}
where $C_s$ is some constant. Thus, under the conditions of Theorem \ref{thm:sparse}, the bound \eqref{equ:stablerecovery} for the springback-penalized model \eqref{equ:CSproblem} is tighter than \eqref{equ:linear} in the sense of
\begin{equation}\label{equ:tighter}
\frac{\sqrt{2}}{\sqrt{D_1}}\sqrt{\tau}\leq C_s \tau
\end{equation}
if the level of noise $\tau$ satisfies
\begin{equation}\label{equ:tau1}
\tau > \frac{2}{D_1C_s^2}.
\end{equation}
Assume that the recovery conditions listed in Table \ref{tab:exact} are satisfied for each model, respectively. Then, we can summarize their corresponding ranges of $\tau$ in Table \ref{tab:robust} such that the robust recovery bound \eqref{equ:stablerecovery} of the springback-penalized model \eqref{equ:CSproblem} is tighter than all the others in the sense of \eqref{equ:tighter}.
\begin{table}[htbp]
  \centering
  \setlength{\abovecaptionskip}{0pt}
\setlength{\belowcaptionskip}{10pt}
  \caption{Ranges of the level of noise such that the springback bound \eqref{equ:stablerecovery} is tighter than the bound \eqref{equ:linear} in the sense of \eqref{equ:tighter}.}\label{tab:robust}
  \begin{tabular}{c|c}
  \hline
    Penalty  & When the springback bound \eqref{equ:stablerecovery} is tighter than the bound \eqref{equ:linear}    \\ \hline
    $\ell_1$  \cite{candes2006stable,candes2005error}  & $\tau > \frac{(\sqrt{3s}+\sqrt{s})(\sqrt{3}\sqrt{1-\delta_{4s}}-\sqrt{1+\delta_{3s}})^2}{4\alpha(\sqrt{1+\delta_{3s}}+\sqrt{1-\delta_{4s}})}$ \\ \hline
    $\ell_p$  ($0<p<1$)  \cite{saab2008stable} & $\tau>\frac{(\sqrt{3s}+\sqrt{s})\left((1-\delta_{4s})^{p/2}-(1+\delta_{3s})^{p/2}3^{p/2-1}\right)^{2/p}}{\alpha(\sqrt{1-\delta_{4s}}+\sqrt{1+\delta_{3s}})\left(1+\frac{1}{(2/p-1)3^{2/p-1}}\right)}$ \\ \hline
    transformed $\ell_1$ \cite{zhang2018minimization} & $\tau>\frac{4(\sqrt{3s}+\sqrt{s})(1-\delta_{3s})\left(\frac{\beta}{\beta+1}\sqrt{3}\sqrt{1-\delta_{4s}}-\sqrt{1+\delta_{3s}}\right)^2}{\alpha(\sqrt{1-\delta_{4s}}+\sqrt{1+\delta_{3s}})\left(\frac{\beta}{\beta+1}\sqrt{3}\sqrt{1-\delta_{4s}}-\sqrt{1+\delta_{3s}}+\sqrt{3s}\sqrt{1-\delta_{3s}}\right)^2}$ \\ \hline
    $\ell_{1-2}$ \cite{yan2017sparse} & $\tau>\frac{(\sqrt{3s}+\sqrt{s})(\sqrt{a(s)(1-\delta_{4s})}-\sqrt{1+\delta_{3s}})^2}{\alpha(\sqrt{1-\delta_{4s}}+\sqrt{1+\delta_{3s}})(\sqrt{3s}-\sqrt{s\cdot a(s)})^2}$  \\
    \hline
  \end{tabular}
\end{table}

These ranges on $\tau$ look complicated. To have a better idea, we consider a toy example with $s=20$, $\delta_{3s}=1/4$, $\delta_{4s}=1/3$, $\alpha=1$ for the spingback penalty \eqref{equ:spb}, and $\beta=1$ for the transformed $\ell_1$ penalty. Then, the springback-penalized model \eqref{equ:CSproblem} would give a tighter bound in the sense of \eqref{equ:tighter} than the $\ell_1$-, $\ell_{0.2}$-, $\ell_{0.5}$-, $\ell_{0.999}$-, transformed $\ell_1$-, and  $\ell_{1-2}$-penalized models if $\tau>0.1385,~   0.0271,~0.2333,~0.1391,~0.0807$, and $2.8652\times 10^{-4}$, respectively.

Can we further improve the robust recovery result \eqref{equ:stablerecovery} in Theorem \ref{thm:sparse}? The following proposition suggests a potential improvement. Moreover, without any requirement on $\alpha$, this proposition also means, even if the posterior verification \eqref{equ:alpha} is violated sometimes, the springback-penalized model \eqref{equ:CSproblem} may still give a good recovery. Note that this proposition is only of conceptual sense, because its assumption $\left<x^{\rm{opt}},x^{\rm{opt}}-\bar{x}\right>\leq 0$ is not verifiable. Nevertheless, it helps us discern a possibility of achieving a better recovery bound than \eqref{equ:stablerecovery}.

\begin{proposition}\label{thm:sparseimproved}
Let $\bar{x}\in\mathbb{R}^n$ be an unknown $s$-sparse vector to be recovered. For a given sensing matrix $A\in\mathbb{R}^{m\times n}$, let $b\in\mathbb{R}^m$ be a vector of measurements from $b=A\bar{x}+e$ with $\|e\|_2\leq\tau$, and let $\delta_{3s}$ and $\delta_{4s}$ be the $3s$- and $4s$-RIC's of $A$, respectively. Let $x^{\rm{opt}}$ be the minimizer of the problem \eqref{equ:CSproblem} and assume $\left<x^{\rm{opt}},x^{\rm{opt}}-\bar{x}\right>\leq 0$. Suppose $A$ satisfies $\delta_{3s}<3(1-\delta_{4s})-1$, then $x^{\rm{opt}}=\bar{x}$ when $\tau=0$; and $x^{\rm{opt}}$ satisfies
\begin{equation}\label{equ:stablerecoveryimproved}
\|x^{\rm{opt}}-\bar{x}\|_2\leq \sqrt{\frac{D_2^2}{4D_1^2}+\frac{2}{D_1}\tau}-\frac{D_2}{2D_1}
\end{equation}
 when $\tau\geq0$, where $D_1$ is the constant \eqref{equ:D1} given in Theorem \ref{thm:sparse} and
\begin{equation}\label{equ:D2}
 D_2=\frac{\sqrt{3}\sqrt{1-\delta_{4s}}-\sqrt{1+\delta_{3s}}}{\sqrt{3}+1}.
\end{equation}
\end{proposition}

\begin{proof}
In the case of $\left<x^{\rm{opt}},v\right>\leq 0$, it follows straightforwardly from \eqref{equ:Av} that
    \begin{equation*}\begin{split}
    \|Av\|_2
    \geq&\frac{\alpha}{2}\left(\frac{\sqrt{1-\delta_{4s}}+\sqrt{1+\delta_{3s}}}{\sqrt{3s}+\sqrt{s}}  \right)\|v\|_2^2+\frac{\sqrt{1-\delta_{4s}}\sqrt{3s}-\sqrt{1+\delta_{3s}}\sqrt{s}}{\sqrt{3s}+\sqrt{s}}\|v\|_2\\
    :=&D_1\|v\|_2^2+D_2\|v\|_2.
    \end{split}\end{equation*}
    The assumption $\delta_{3s}<3(1-\delta_{4s})-1$ guarantees $D_2>0$. Hence, when $\tau=0$, as $Av=A(x^{\rm{opt}}-\bar{x})=0$, we have $0=\|Av\|_2\geq D_1\|v\|_2^2+D_2\|v\|_2$, which implies $\|v\|_2=0$. When $\tau>0$, the inequality
    $$
    \|Av\|_2=\|Ax^{\rm{opt}}-A\bar{x}\|_2\leq\|Ax^{\rm{opt}}-b\|_2+\|A\bar{x}-b\|_2\leq2\tau
    $$
    implies
    \begin{equation*}
    \|v\|_2\leq\frac{\sqrt{D_2^2+8D_1\tau}-D_2}{2D_1}.
    \end{equation*}
The assertion is proved.
\end{proof}

\begin{remark}
The robust recovery result \eqref{equ:stablerecoveryimproved} is always better than \eqref{equ:stablerecovery} in Theorem \ref{thm:sparse} due to the subadditivity of the square root function. Under the conditions of Proposition \ref{thm:sparseimproved}, the bound \eqref{equ:stablerecoveryimproved} for the springback-penalized model \eqref{equ:CSproblem} is tighter than \eqref{equ:linear} in the sense of
\begin{equation*}
\sqrt{\frac{D_2^2}{4D_1^2}+\frac{2}{D_1}\tau}-\frac{D_2}{2D_1}<C_s\tau,
\end{equation*}
if the level of noise $\tau$ satisfies
\begin{equation*}\label{equ:tau2}
\tau > \frac{2-D_2C_s}{D_1C_s^2}=\left(1-\frac{D_2C_s}{2}\right)\frac{2}{D_1C_s^2}.
\end{equation*}
Comparing with \eqref{equ:tau1}, this improvement enlarges the value range of $\tau$. For example, if $C_s$ is the coefficient in the result \eqref{equ:l1sparserecovery} of the BP model \eqref{equ:l1} , then $1-D_2C_s/{2}$ is approximately 0.2679.
\end{remark}

\section{Springback-penalized model for nearly sparse signal recovery}\label{sec:signal2}

We then study the stable recovery of the springback-penalized model \eqref{equ:CSproblem} when $\bar{x}$ is nearly sparse and the measurements are noisy.

\subsection{Recovery guarantee using the springback-penalized model}
If the signal $\bar{x}$ to be recovered is nearly $s$-sparse, then we have the following stable recovery theorem for the springback-penalized model \eqref{equ:CSproblem}.

\begin{theorem}[recovery of nearly sparse signals]\label{thm:nonsparse}
Let $\bar{x}\in\mathbb{R}^n$ be an unknown vector to be recovered. For a given sensing matrix $A\in\mathbb{R}^{m\times n}$, let $b\in\mathbb{R}^m$ be a vector of measurements from $b=A\bar{x}+e$ with $\|e\|_2\leq\tau$, and let $\delta_{3s}$ and $\delta_{4s}$ be the $3s$- and $4s$-RIC's of $A$, respectively. Let $\bar{x}_s\in\mathbb{R}^n$ be the truncated vector corresponding to the $s$ largest values of $\bar{x}$ (in absolute value). Suppose $A$ satisfies $\delta_{3s}<3(1-\delta_{4s})-1$ and $\alpha$ satisfies \eqref{equ:alpha}, then the minimizer $x^{\rm{opt}}$ of the problem \eqref{equ:CSproblem} satisfies
\begin{equation}\label{equ:stablerecoverynonsparse}
\|x^{\rm{opt}}-\bar{x}\|_2\leq \sqrt{\frac{2}{D_1}\tau+\frac{4}{\alpha}\|\bar{x}-\bar{x}_s\|_1},
\end{equation}
where $D_1$ is the constant \eqref{equ:D1} given in Theorem \ref{thm:sparse}.
\end{theorem}

\begin{proof}
Let $x^{\rm{opt}}=\bar{x}+v$, and $\Lambda_0$ be the support of $\bar{x}_s$. It is clear that $v_{\Lambda_0}=x^{\rm{opt}}_{\Lambda_0}-\bar{x}_s$ and $v_{\Lambda_0^c}=x^{\rm{opt}}_{\Lambda_0^c}-\bar{x}_{\Lambda_0^c}$. We know that
\begin{equation*}
\|x^{\rm{opt}}\|_1-\frac{\alpha}{2}\|x^{\rm{opt}}\|_2^2\leq\|\bar{x}\|_1-\frac{\alpha}{2}\|\bar{x}\|_2^2=\|\bar{x}_s\|_1+\|\bar{x}_{\Lambda_0^c}\|_1-\frac{\alpha}{2}\|\bar{x}\|_2^2.
\end{equation*}
On the other hand, it holds that
\begin{equation*}\begin{split}
\|x^{\rm{opt}}\|_1-\frac{\alpha}{2}\|x^{\rm{opt}}\|_2^2
&=\|\bar{x}_s+v_{\Lambda_0}\|_1+\|\bar{x}_{\Lambda_0^c}+v_{\Lambda_0^c}\|_1-\frac{\alpha}{2}\|\bar{x}+v\|_2^2\\
&\geq \|\bar{x}_s\|_1-\|v_{\Lambda_0}\|_1+\|v_{\Lambda_0^c}\|_1-\|\bar{x}_{\Lambda_0^c}\|_1-\frac{\alpha}{2}\left(\|\bar{x}\|_2^2+2\left<\bar{x},v\right>+\|v\|_2^2\right).
\end{split}\end{equation*}
Then, $v$ satisfies the following estimation:
\begin{equation*}\begin{split}
\|v_{\Lambda_0^c}\|_1
\leq& \|v_{\Lambda_0}\|_1+2\|\bar{x}-\bar{x}_s\|_1-\frac{\alpha}{2}\|v\|_2^2+\alpha\|v\|_2^2+\alpha\left<\bar{x},v\right>\\
=& \|v_{\Lambda_0}\|_1+2\|\bar{x}-\bar{x}_s\|_1-\frac{\alpha}{2}\|v\|_2^2+\alpha\left<x^{\rm{opt}},v\right>.
\end{split}\end{equation*}

We divide $\Lambda_0^c$ into subsets of size $3s$, $\Lambda_0^c=\Lambda_1\bigcup\Lambda_2\bigcup\cdots\bigcup\Lambda_{\ell}$, in terms of decreasing order of magnitudes (in absolute value) of $v_{\Lambda_0^c}$. Denoting $\Lambda_{01}=\Lambda_0\bigcup\Lambda_1$ and using the RIP of A, we have
\begin{equation*}
\|Av\|_2\geq\|A_{\Lambda_{01}}v_{\Lambda_{01}}\|_2-\left\|\sum_{i=2}^{\ell}A_{\Lambda_i}v_{\Lambda_i}\right\|_2\geq\sqrt{1-\delta_{4s}}\|v_{\Lambda_{01}}\|_2-\sqrt{1+\delta_{3s}}\sum_{i=2}^{\ell}\|v_{\Lambda_i}\|_2.
\end{equation*}
As proved for  Theorem \ref{thm:sparse}, we have $\sum_{i=2}^{\ell}\|v_{\Lambda_i}\|_2\leq\|v_{\Lambda_0^c}\|_1/\sqrt{3}$ and $\|v_{\Lambda_0}\|_1\leq\sqrt{s}\|v_{\Lambda_{01}}\|_2$. Thus, we obtain
\begin{equation*}
\sum_{i=2}^{\ell}\|v_{\Lambda_i}\|_2\leq\frac{1}{\sqrt{3s}}\left(\sqrt{s}\|v_{\Lambda_{01}}\|_2+2\|\bar{x}-\bar{x}_s\|_1-\frac{\alpha}{2}\|v\|_2^2+\alpha\left<x^{\rm{opt}},v\right>\right).
\end{equation*}
Furthermore, it holds that
\begin{equation}\label{equ:Avcase3}\begin{split}
\|Av\|_2\geq&\left(\sqrt{1-\delta_{4s}}-\frac{\sqrt{1+\delta_{3s}}\sqrt{s}}{\sqrt{3s}}\right)\|v_{\Lambda_{01}}\|_2-\frac{2\sqrt{1+\delta_{3s}}}{\sqrt{3s}}\|\bar{x}-\bar{x}_s\|_1\\
&+\frac{\alpha\sqrt{1+\delta_{3s}}}{2\sqrt{3s}} \|v\|_2^2-\frac{\alpha\sqrt{1+\delta_{3s}}}{\sqrt{3s}}\left<x^{\rm{opt}},v\right>.
\end{split}\end{equation}
As
\begin{equation*}\begin{split}
\|v\|_2&\leq\|v_{\Lambda_{01}}\|_2+\sum_{i=2}^{\ell}\|v_{\Lambda_i}\|_2\\
&\leq\left(1+\frac{\sqrt{s}}{\sqrt{3s}}\right)\|v_{\Lambda_{01}}\|_2+\frac{2}{\sqrt{3s}}\|\bar{x}-\bar{x}_s\|_1-\frac{\alpha}{2\sqrt{3s}} \|v\|_2^2+\frac{\alpha}{\sqrt{3s}}\left<x^{\rm{opt}},v\right>,
\end{split}\end{equation*}
we have
\begin{equation*}
\|v_{\Lambda_{01}}\|_2\geq\frac{\sqrt{3s}}{\sqrt{3s}+\sqrt{s}}\left(\frac{\alpha}{2\sqrt{3s}}\|v\|_2^2+\|v\|_2-\frac{\alpha}{\sqrt{3s}}\left<x^{\rm{opt}},v\right>-\frac{2}{\sqrt{3s}}\|\bar{x}-\bar{x}_s\|_1\right).
\end{equation*}
Recall the assumption $\delta_{3s}<3(1-\delta_{4s})-1$. The coefficient of $\|v_{\Lambda_{01}}\|_2$ in \eqref{equ:Avcase3} is positive, and it follows that
\begin{equation}\label{equ:Av2}\begin{split}
\|Av\|_2
\geq& \frac{\sqrt{1-\delta_{4s}}\sqrt{3s}-\sqrt{1+\delta_{3s}}\sqrt{s}}{\sqrt{3s}+\sqrt{s}}\left(\frac{\alpha}{2\sqrt{3s}}\|v\|_2^2+\|v\|_2  -\frac{\alpha}{\sqrt{3s}}\left<x^{\rm{opt}},v\right>  -\frac{2}{\sqrt{3s}}\|\bar{x}-\bar{x}_s\|_1\right) \\
&+\frac{\alpha\sqrt{1+\delta_{3s}}}{2\sqrt{3s}} \|v\|_2^2-\frac{\alpha\sqrt{1+\delta_{3s}}}{\sqrt{3s}}\left<x^{\rm{opt}},v\right>-\frac{2\sqrt{1+\delta_{3s}}}{\sqrt{3s}}\|\bar{x}-\bar{x}_s\|_1\\
=&\frac{\alpha}{2}\left(\frac{\sqrt{1-\delta_{4s}}+\sqrt{1+\delta_{3s}}}{\sqrt{3s}+\sqrt{s}}  \right)\|v\|_2^2+\frac{\sqrt{1-\delta_{4s}}\sqrt{3s}-\sqrt{1+\delta_{3s}}\sqrt{s}}{\sqrt{3s}+\sqrt{s}}\|v\|_2\\
    &-\alpha\left(\frac{\sqrt{1-\delta_{4s}}+\sqrt{1+\delta_{3s}}}{\sqrt{3s}+\sqrt{s}}  \right)\left<x^{\rm{opt}},v\right>-2\left(\frac{\sqrt{1-\delta_{4s}}+\sqrt{1+\delta_{3s}}}{\sqrt{3s}+\sqrt{s}}  \right)\|\bar{x}-\bar{x}_s\|_1.
\end{split}\end{equation}
If $\left<x^{\rm{opt}},v\right>\leq0$, then $\|Av\|_2\geq D_1\|v\|_2^2-\frac{4}{\alpha}D_1\|\bar{x}-\bar{x}_s\|_1$. If $\left<x^{\rm{opt}},v\right>>0$, then the condition \eqref{equ:alpha} on $\alpha$ guarantees
\begin{equation*}\begin{split}
\frac{\sqrt{1-\delta_{4s}}\sqrt{3s}-\sqrt{1+\delta_{3s}}\sqrt{s}}{\sqrt{3s}+\sqrt{s}}\|v\|_2-\alpha\left(\frac{\sqrt{1-\delta_{4s}}+\sqrt{1+\delta_{3s}}}{\sqrt{3s}+\sqrt{s}}  \right)\left<x^{\rm{opt}},v\right>\geq 0,
\end{split}\end{equation*}
which is shown in the proof of Theorem \ref{thm:sparse}. Hence, we also have $\|Av\|_2\geq D_1\|v\|_2^2-\frac{4}{\alpha}D_1\|\bar{x}-\bar{x}_s\|_1$.
As $\|Av\|_2=\|Ax^{\rm{opt}}-A\bar{x}\|_2\leq\|Ax^{\rm{opt}}-b\|_2+\|A\bar{x}-b\|_2\leq2\tau$, we have
\begin{equation*}
2\tau\geq D_1\|v\|_2^2-\frac{4}{\alpha}D_1\|\bar{x}-\bar{x}_s\|_1,
\end{equation*}
which implies \eqref{equ:stablerecoverynonsparse}.
\end{proof}

Similar to the improvement in Proposition \ref{thm:sparseimproved}, the above stable recovery result can be improved as follows.
\begin{proposition}\label{thm:nonsparseimproved}
Let $\bar{x}\in\mathbb{R}^n$ be an unknown vector to be recovered. For a given sensing matrix $A\in\mathbb{R}^{m\times n}$, let $b\in\mathbb{R}^m$ be a vector of measurements from $b=A\bar{x}+e$ with $\|e\|_2\leq\tau$, and let $\delta_{3s}$ and $\delta_{4s}$ be the $3s$- and $4s$-RIC's of $A$, respectively. Let $x^{\rm{opt}}$ be the minimizer of the problem \eqref{equ:CSproblem} and assume $\left<x^{\rm{opt}},x^{\rm{opt}}-\bar{x}\right>\leq 0$. Let $\bar{x}_s\in\mathbb{R}^n$ be the truncated vector corresponding to the $s$ largest values of $\bar{x}$ (in absolute value). Suppose $A$ satisfies $\delta_{3s}<3(1-\delta_{4s})-1$,
then $x^{\rm{opt}}$ satisfies
\begin{equation*}
\|x^{\rm{opt}}-\bar{x}\|_2\leq\sqrt{\frac{D_2^2}{4D_1^2}+\frac{2}{D_1}\tau+\frac{4}{\alpha}\|\bar{x}-\bar{x}_s\|_1}-\frac{D_2}{2D_1},
\end{equation*}
where $D_1$ and $D_2$ are the constants \eqref{equ:D1} and \eqref{equ:D2} given in Theorem \ref{thm:sparse} and Proposition \ref{thm:sparseimproved}, respectively.
\end{proposition}

\begin{proof}
In the case of $\left<x^{\rm{opt}},v\right>\leq 0$, it follows straightforwardly from the estimation \eqref{equ:Av2} that
    \begin{equation*}\begin{split}
    \|Av\|_2\geq D_1\|v\|_2^2+D_2\|v\|_2-\frac{4}{\alpha}D_1\|\bar{x}-\bar{x}_s\|_1.
    \end{split}\end{equation*}
    The assumption $\delta_{3s}<3(1-\delta_{4s})-1$ guarantees $D_2>0$. Therefore, it follows from the triangle inequality that
     \begin{equation*}
     \|Av\|_2=\|Ax^{\rm{opt}}-A\bar{x}\|_2\leq\|Ax^{\rm{opt}}-b\|_2+\|A\bar{x}-b\|_2\leq2\tau.
     \end{equation*}
We thus have
    \begin{equation}\label{equ:ineq}
    D_1\|v\|_2^2+D_2\|v\|_2-\frac{4}{\alpha}D_1\|\bar{x}-\bar{x}_s\|_1\leq2\tau,
    \end{equation}
    which gives the improved result by solving the system of inequalities \eqref{equ:ineq} and $\|v\|_2\geq 0$.
\end{proof}

\subsection{On the stable recovery}\label{sec:stable3}
If $\bar{x}$ is known to be $s$-sparse, then the estimation \eqref{equ:stablerecoverynonsparse} in Theorem \ref{thm:nonsparse} is reduced to \eqref{equ:stablerecovery} in Theorem \ref{thm:sparse}; and if the measurements are additionally noise-free, then both the estimations \eqref{equ:stablerecovery} and \eqref{equ:stablerecoverynonsparse} imply exact recovery of the signal $\bar{x}$. We compare the estimation \eqref{equ:stablerecoverynonsparse} with the estimation \eqref{equ:l1nonsparserecovery} for the BP model \eqref{equ:l1}. The following comparison is based on theoretical error bounds. We are interested in the case where the estimation \eqref{equ:stablerecoverynonsparse} is tighter than the estimation \eqref{equ:l1nonsparserecovery} in the sense of
\begin{equation}\label{equ:comp}
\sqrt{\frac{2}{D_1}\tau+\frac{4}{\alpha}\|\bar{x}-\bar{x}_s\|_1}\leq C_{1,s}\tau+C_{2,s}\frac{\|\bar{x}-\bar{x}_{s}\|_1}{\sqrt{s}},
\end{equation}
which is equivalent to
\begin{equation}\label{equ:sss}
\frac{s^{1/4}}{\sqrt{\alpha}}\sqrt{\frac{4(\sqrt3+1)}{\sqrt{1-\delta_{4s}}+\sqrt{1+\delta_{3s}}}\tau+\frac{4\|\bar{x}-\bar{x}_{s}\|_1}{\sqrt{s}}}\leq C_{1,s}\tau+C_{2,s}\frac{\|\bar{x}-\bar{x}_{s}\|_1}{\sqrt{s}}.
\end{equation}
Note that $s$ takes values among $\{1,2,\ldots,n\}$ and the right-hand side of \eqref{equ:sss} decreases as $s$ increases. If the left-hand side of \eqref{equ:sss} is smaller than the right-hand side of \eqref{equ:sss} for $s=1$ and the left-hand side is larger than the right-hand side for $s=n$, then there must exist a constant $C$ such that the inequality \eqref{equ:comp} holds for $s\leq C$. Besides, if $\bar{x}$ is known to be $s$-sparse, then $\|\bar{x}-\bar{x}_s\|_1=0$ and thus \eqref{equ:sss} implies the existence of $C$ without any assumption. Therefore, we have the following corollary.

\begin{corollary}\label{coro}
If $\bar{x}$ is $s$-sparse, then there exists a constant $C$ such that the inequality \eqref{equ:comp} holds for $s\leq C$, where
\begin{equation}\label{equ:C}
C = \alpha^2 C_{1,s}^4\tau^2\left(\frac{\sqrt{1-\delta_{4s}}+\sqrt{1+\delta_{3s}}}{4(\sqrt3+1)}\right)^2.
\end{equation}
When no information of the sparsity of $\bar{x}$ is known, if $\alpha$ satisfies
\begin{equation*}
\frac{\frac{4(\sqrt3+1)}{\sqrt{1-\delta_{4}}+\sqrt{1+\delta_{3}}}\tau+4\|\bar{x}-\bar{x}_{1}\|_1}{(C_{1,1}\tau+C_{2,1}\|\bar{x}-\bar{x}_{1}\|_1)^2}\leq\alpha\leq\frac{1}{C_{1,n}^2\tau}\frac{4(\sqrt3+1)\sqrt{n}}{\sqrt{1-\delta_{4n}}+\sqrt{1+\delta_{3n}}},
\end{equation*}
then there exists a constant $C$ such that the inequality \eqref{equ:comp} holds for $s\leq C$, where $C$ depends on $\alpha$, $\bar{x}$, $\tau$, $\delta_{3s}$, and $\delta_{4s}$.
\end{corollary}

In virtue of random matrix theory, we give two examples to show that the condition $\delta_{3s}<3(1-\delta_{4s})-1$ on $A$ in Theorems \ref{thm:sparse} and \ref{thm:nonsparse} holds.
\begin{itemize}
    \item \emph{Random Gaussian matrices}: the entries of $A$ are i.i.d. Gaussian with mean zero and variance $1/m$. It was shown in \cite{candes2005error,candes2005decoding} that the condition $\delta_{3s}<3(1-\delta_{4s})-1$ holds with overwhelming probability when $s\leq C' m/\log(n/m)$, where $C'$ is a constant. Similar results were extended to sub-gaussian matrices in \cite{mendelson2007reconstruction}.
    \item \emph{Fourier ensemble}: $A$ is obtained by selecting $m$ rows from the $n\times n$ discrete Fourier transform and renormalizing the columns so that they are unit-normed. If the rows are selected at random, the condition $\delta_{3s}<3(1-\delta_{4s})-1$ holds with overwhelming probability for $s\leq C'm/(\log(n))^4$, where $C'$ is a constant. This was initially considered in \cite{candes2006near} and then improved in \cite{rudelson2008sparse}.
\end{itemize}

\begin{remark}
Assume that $\alpha$ satisfies the conditions in Theorem \ref{thm:nonsparse} and Corollary \ref{coro}. For a random Gaussian sensing matrix $A$, if $s\leq C' m/\log(n/m)$, then the RIP condition $\delta_{3s}<3(1-\delta_{4s})-1$ on $A$ holds with high probability; and additionally if, $C' m/\log(n/m)\leq C$, i.e.,
\begin{equation*}
m\exp\left(\frac{C'}{C}m\right)\leq n,
\end{equation*}
then the estimation \eqref{equ:stablerecoverynonsparse} is tighter than the estimation \eqref{equ:l1nonsparserecovery} in the sense of \eqref{equ:comp}. For a randomly subsampled Fourier sensing matrix $A$, if $s\leq C'm/(\log(n))^4$, then the RIP condition on $A$ holds with overwhelming probability; and additionally if $C'm/(\log(n))^4\leq C$, i.e.,
\begin{equation*}
m\leq\frac{C}{C'}(\log(n))^4,
\end{equation*}
then the estimation \eqref{equ:stablerecoverynonsparse} is tighter than the estimation \eqref{equ:l1nonsparserecovery} in the sense of \eqref{equ:comp}. In a nutshell, for a sensing matrix satisfying the RIP condition, if the number $m$ of observation data is limited, where ``limited'' can be characterized as the fact that $m$ is less than some constant depending on $n$, $C$, and $C'$, then the stable recovery using the springback-penalized model \eqref{equ:CSproblem} is guaranteed by a tighter bound than that of BP model \eqref{equ:l1} in the sense of \eqref{equ:comp}. These results can be extended to general orthogonal sensing matrices \cite{candes2006near}. Similar comparative results with other recovery models may also be derived if the recovery error bounds of these models are linear to $\tau$ and $\|\bar{x}-\bar{x}_{s}\|_1$, e.g., the $\ell_{1-2}$-penalized model \cite{yan2017sparse}.
\end{remark}

\section{Computational aspects of the springback-penalized model}\label{sec:computing}

Now we focus on computational aspects for the springback-penalized model \eqref{equ:CSproblem}. We first design an algorithm for solving \eqref{equ:CSproblem} in Section \ref{sec:dca}, and then discuss its convergence in Section \ref{sec:convergence} and elaborate on how to solve its subproblems in Section \ref{sec:subproblem}.

\subsection{DCA-springback: An algorithm for the springback penalized model}\label{sec:dca}

Some well-developed algorithms for solving difference-of-convex (DC) optimization problems can be easily implemented to solve the springback-penalized model \eqref{equ:CSproblem}. We focus on the simplest DCA in \cite{tao1997convex,tao1998dc} without any line-search step, which has been shown to be efficient for solving signal recovery problems, see, e.g., \cite{MR4201341,yin2015minimization,zhang2018minimization}.

Recall a standard DC optimization problem
\begin{equation}\label{equ:standardDC}
\min_{x\in\mathbb{R}^n}~f(x):=g(x)-h(x),
\end{equation}
where $g$ and $h$ are lower semicontinuous proper convex functions on $\mathbb{R}^n$. Here, $f$ is called a DC function, and $g-h$ is a DC decomposition of $f$. At each iteration, the DCA replaces the concave part $-h$ with a linear majorant and solves the resulting convex problem. That is, the DCA generates a sequence $\{x^k\}$ by solving the following subproblem iteratively:
\begin{equation*}
x^{k+1}\in\arg\min_{x\in\mathbb{R}^n}~ \left\{g(x)-\langle x-x^{k},\xi^k \rangle\right\},
\end{equation*}
where $\xi^{k}\in\partial (h(x^k))$. Note that the springback-penalized model \eqref{equ:CSproblem} can be written as
\begin{equation}\label{equ:convexconsDC}
\arg\min_{x\in\mathbb{R}^n} F(x):=\left(\|x\|_1 + \chi_{\Omega}(x)\right)-\frac{\alpha}{2}\|x\|_2^2,
\end{equation}
where $\Omega:=\{x\in\mathbb{R}^n:\|Ax-b\|_2\leq \tau\}$ and
\begin{equation*}
\chi_{\Omega}(x):=\begin{cases}
0, & x\in \Omega,\\
+\infty, &x\not\in \Omega,
\end{cases}
\end{equation*}
is the indictor function of the set $\Omega$. Thus, the DCA iterate scheme for solving \eqref{equ:convexconsDC} reads as
\begin{equation*}\begin{split}
x^{k+1}&\in \arg\min_{x} \left\{\left(\|x\|_1+\chi_{\Omega}(x)\right) - \langle x-x^{k},\xi^k \rangle\right\} \\
&=\arg\min_{x}~ \left\{ \|x\|_1-\langle x-x^{k},\xi^k \rangle\quad\text{s.t.}\quad x\in \Omega\right\}.
\end{split}\end{equation*}
More specifically, the resulting DCA is listed in Algorithm \ref{alg:DCA}, where $\epsilon_{\text{outer}}>0$ is the preset tolerance for iterations, and ``MaxIt" means the maximal number of iterations set beforehand.

\begin{algorithm}[htbp]\label{alg:DCA}
\caption{DCA-springback: Solving the constrained springback model \eqref{equ:CSproblem} via DCA}
\KwIn{Model parameters: $\alpha>0$ satisfying the condition \eqref{equ:alpharequirement};\\
Stopping criterion: $\epsilon_{\rm{outer}}>0$, ${\rm{MaxIt}>0}$;\\
Initialization: $k=0$, $x^0$ satisfying $\|Ax-b\|_2\leq \tau$;}
\While{$k<{\rm{MaxIt}}$ {\rm{and}} $\min\left\{\|x^{k+1}-x^{k}\|_2,\|x^{k+1}-x^{k}\|_2/\|x^k\|_2\right\}>\epsilon_{\rm{outer}}$}
{$\xi^k=\alpha x^k$\;
$x^{k+1}\in\arg\min\limits_{x}~\left\{\|x\|_1-\left<x-x^{k},\xi^k\right>~ \text{s.t.}~ \|Ax-b\|_2\leq \tau\right\}$\;
$k\leftarrow k+1$\;}
\end{algorithm}

\subsection{Convergence}\label{sec:convergence}
Recall that the \emph{modulus of strong convexity} of a convex function $f$ on $\mathbb{R}^n$, denoted by $d(f)$, is defined as $d(f):=\sup\{\nu>0:f(\cdot)-\frac{\nu}{2}\|\cdot\|_2^2\text{ is convex on }\mathbb{R}^n\}$. Then, according to \cite[Proposition A.1]{tao1998dc}, for a general DC function $f=g-h$, any sequence $\{x^k\}$ generated by the DCA satisfies
\begin{equation}\label{equ:DCAconvergence}
f(x^k)-f(x^{k+1})\geq\frac{d(g)+d(h)}{2}\|x^{k+1}-x^k\|_2^2,
\end{equation}
which immediately implies the decreasing property of $\{f(x^k)\}$ if at least one of $g$ and $h$ is strongly convex. Note that $\frac{\alpha}{2}\|x\|_2^2$ is strongly convex with modulus $\alpha$. Thus, starting with a feasible $x^0$, we have the decreasing property
\begin{equation}\label{equ:Fdecreasing}
F(x^k)- F(x^{k+1})\geq \frac{\alpha}{2}\|x^{k+1}-x^k\|_2^2,
\end{equation}
where $F$ is defined as \eqref{equ:convexconsDC}. However, the decreasing property \eqref{equ:Fdecreasing} of $F$ is not sufficient to ensure the convergence of DCA-springback. The function $F$ could be negative if $\alpha$ is inappropriately large. Note that for any $x^k$, we have
$$\|Ax^k\|_2- \|b\|_2\leq \|Ax^k-b\|_2\leq \tau.$$
Moreover, as $A$ is assumed to be full rank, we have $\sigma_{\min}(A)>0$. It follows from the geometric interpretation of the SVD \cite[Lecture 4]{MR1444820} that $\|Ax\|_2\geq\sigma_{\min}(A)$ for any $x\in\mathbb{R}^n$ on the unit sphere $\{x\in\mathbb{R}^n:\|x\|_2=1\}$. Thus, it holds that
$$0<\sigma_{\min}(A)\leq \min_{x\in\mathbb{R}^n}\frac{\|Ax\|_2}{\|x\|_2} = \min_{\|x\|_2=1}\|Ax\|_2,$$
and we have
\begin{equation}\label{equ:xkestimate}
\|x^k\|_2 \leq \frac{\|b\|_2+\tau}{\sigma_{\min}(A)}.
\end{equation}
Note that $\|x\|_1-\frac{\alpha}{2}\|x\|_2^2 \geq0$ and hence $F$ is non-negative if $\|x\|_2\leq 2/\alpha$. Clearly, if
\begin{equation}\label{equ:alpharequirement}
\alpha \leq \frac{2\sigma_{\min}(A)}{\|b\|_2+\tau},
\end{equation}
then $F(x^k)\geq 0$ for any $k\geq 0$ because all iterates $x^k$ satisfy \eqref{equ:xkestimate}. Together with the decreasing property \eqref{equ:Fdecreasing}, we can establish the convergence of DCA-springback easily by following the analytical framework in \cite{tao1997convex,tao1998dc}. Moreover, it follows the convergence of $\{F(x^k)\}$ and \eqref{equ:Fdecreasing} that  $\|x^{k+1}-x^k\|_2\rightarrow 0$ as $k\rightarrow \infty$.

\begin{remark}
Note that the condition \eqref{equ:alpha} depends on the RIP condition of $A$, and \eqref{equ:alpharequirement} depends on the conditioning of $A$. It is easy to deduce that if
\begin{equation}\label{equ:alphasituation2}
\frac{\sqrt{1-\delta_{4s}}\sqrt{3s}-\sqrt{1+\delta_{3s}}\sqrt{s}}{\sqrt{1-\delta_{4s}}+\sqrt{1+\delta_{3s}}}\leq \frac{2\sigma_{\min}(A)\|x^{\rm{opt}}\|_2}{\|b\|_2+\tau},
\end{equation}
then the condition \eqref{equ:alpharequirement} is implied by \eqref{equ:alpha}. Otherwise, it can be verified that the condition \eqref{equ:alpha} is implied by \eqref{equ:alpharequirement}.
\end{remark}

\subsection{Solving the subproblem of DCA-springback}\label{sec:subproblem}

For the proposed DCA-springback, its subproblem at each iteration is
\begin{equation}\label{equ:sub}
\min\limits_{x}~\|x\|_1-\left<x-x^{k},\xi^k\right>\quad \text{s.t.}\quad \|Ax-b\|_2\leq \tau.
\end{equation}
This problem can be easily solved by, e.g., the ADMM, which was originally proposed in \cite{MR388811} and had been well developed in the literature such as \cite{chen2016direct,he20121}. Some details are given for completeness. Note that the subproblem \eqref{equ:sub} can be reformulated as
\begin{equation*}\begin{split}
\min_{x,y,z}&~\|y\|_1-\langle x-x^{k},\xi^k\rangle\\
{\rm{s.t.}}&~ y=x,\\
&~z = Ax-b,\\
&~z \in \mathcal{B}(\tau),
\end{split}\end{equation*}
where $y,z\in\mathbb{R}^n$ are two auxiliary variables. With some trivial details skipped, the iterative scheme of the (scaled) ADMM  for the subproblem \eqref{equ:sub} reads as
\begin{equation}\label{equ:ADMM}\begin{cases}
x^{j+1}=(\rho A^{\rm{T}}A+\zeta I)^{-1}\left(\rho A^{\rm{T}}(b+z^j-\eta^j)+\xi^k+\zeta (y^j-u^j)\right),\\
y^{j+1}_i = {\rm{soft}}(x^{j+1}_i+u^j_i;1/\zeta) \text{ for } i=1,\ldots,n,\\
z^{j+1} = \mathcal{P}_{\mathcal{B}(\tau)}(Ax^{j+1}-b+\eta^j),\\
u^{j+1}=u^j+x^{j+1}-y^{j+1},\\
\eta^{j+1} = \eta^{j+1}+Ax^{j+1}-b-z^{j+1},
\end{cases}\end{equation}
where $u\in\mathbb{R}^n$ and $\eta \in\mathbb{R}^m$ are the Lagrange multipliers, $\zeta>0$ and $\rho>0$ are penalty parameters, and $\mathcal{P}_{\mathcal{B}(\tau)}(\cdot)$ is the projection operator onto the ball $\mathcal{B}(\tau)$.
If the measurement process is noise-free, i.e., $\tau=0$, then $z^j$ is always set as zero and the projection of the $z$-subproblem  in \eqref{equ:ADMM} is not necessary.

\section{Numerical experiments}\label{sec:numerical}
In this section, we implement the DCA-springback to the constrained springback-penalized model \eqref{equ:CSproblem} with simulated data. All codes were written by MATLAB R2022a, and all numerical experiments were conducted on a laptop (16 GB RAM, Intel® CoreTM i7-9750H Processor) with macOS Monterey 12.4.

We mainly show the effectiveness of the model \eqref{equ:CSproblem} for some specific scenarios and demonstrate the efficiency of the DCA-springback. Several state-of-the-art signal recovery solvers listed below are also tested for comparison.

\begin{enumerate}
\item [1)] The accelerated iterative hard thresholding (AIHT) algorithm in \cite{blumensath2009iterative}: solving the constrained model
    \begin{equation*}
    \min_{x\in\mathbb{R}^n}~\|Ax-b\|_2^2\quad\text{s.t.}\quad \|x\|_0\leq s
    \end{equation*}
    by the accelerated iterative hard thresholding, where $s$ is set beforehand to estimate the sparsity of $x$. For simplicity, we only choose the fundamental AIHT in  \cite{blumensath2009iterative}, and refer to, e.g., \cite{foucart2011hard,huang2018constructive,MR3398943,jiao2017iterative,needell2009cosamp,pati1993orthogonal}, for various other more sophisticated algorithms.

\item [2)] ADMM-$\ell_1$ \cite{MR388811}: solving the unconstrained $\ell_1$-penalized problem by the ADMM.

\item [3)] IRLS-$\ell_p$ ($0<p<1$) \cite{lai2013improved}: smoothing the unconstrained $\ell_p$-penalized model as
    \begin{equation*}\min_{x\in\mathbb{R}^n}~\frac12\|Ax-b\|_2^2+\lambda\|x\|_{p,\epsilon}^p\quad\text{with}\quad \|x\|_{p,\epsilon}^p:=\sum_{j=1}^n(x_j^2+\epsilon^2)^{p/2},
    \end{equation*}
    where $\epsilon>0$, and implementing the iteratively reweighted least squares (IRLS) algorithm.

\item [4)] DCA-TL1 \cite{zhang2018minimization}: solving the unconstrained transformed $\ell_1$-penalized model with parameter $\beta$ by DCA and implementing the ADMM for its subproblems.

\item [5)] DCA-$\ell_{1-2}$ \cite{yin2015minimization}: solving the unconstrained $\ell_{1-2}$-penalized model by DCA and implementing the ADMM for its subproblems.

\item [6)] DCA-MCP \cite{sun2018sparse}: solving the unconstrained MCP-penalized model by DCA and implementing the ADMM for its subproblems (the authors in \cite{sun2018sparse} consider the $\ell_1$-norm data fidelity term instead of the $\ell_2$ norm, but the implementation of the MCP term is similar).
\end{enumerate}

Note that the AIHT solves the $\ell_0$-penalized model directly; the ADMM-$\ell_1$ solves a convex surrogate model, and the others solve different non-convex approximate models.

\subsection{Setup}
We consider both incoherent and coherent sensing matrices to generate synthetic data for simulation. In the incoherent regime, we use random Gaussian matrices and random partial discrete cosine transform (DCT) matrices. For the former kind, its columns are generated by
\begin{equation*}
A_i\overset{\text{i.i.d.}}{\sim}\mathcal{N}(0,I_m/m),\quad i=1,\ldots,n,
\end{equation*}
where $\mathcal{N}(0,I_m/m)$ is the multivariate Gaussian distribution with location $0$ and covariance $I_m/m$. For the latter kind, its columns are generated by
\begin{equation*}
A_i=\frac{1}{\sqrt{m}}\cos(2i\pi\chi_i),\quad i = 1,\ldots,n,
\end{equation*}
where $\chi_i\in\mathbb{R}^m\overset{\text{i.i.d.}}{\sim}\mathcal{U}([0,1]^m)$ is uniformly and independently sampled from $[0,1]$. Note that both kinds of matrices have small RIP constants with high probability. The coherent regime consists of more ill-conditioned sensing matrices with higher coherence, and it is represented by the randomly oversampled partial DCT matrix in our experiments. A randomly oversampled partial DCT matrix is defined as
\begin{equation*}
        A_i=\frac{1}{\sqrt{m}}\cos(2i\chi_i/\mathcal{F}),\quad i = 1,\ldots,n,
\end{equation*}
where $\mathcal{F}\in\mathbb{N}$ is the \emph{refinement factor}. As $\mathcal{F}$ increases, $A$ becomes more coherent. A matrix sampled in this way cannot satisfy an RIP, and the sparse recovery with such a matrix is possible only if the non-zero elements of the ground-truth $\bar{x}$ are sufficiently separated. Technically, we select the elements of $\text{supp}(\bar{x})$ such that $\min_{j,k\in\text{supp}(\bar{x})}|j-k|\geq L,$ where $L$ is characterized as the \emph{minimum separation}.

We generate a ground-truth vector $\bar{x}\in\mathbb{R}^n$ with sparsity $s$ supported on a random index set (for incoherent matrices) or an index set satisfying the required minimum separation (for coherent matrices) with non-zero entries i.i.d. drawn from the normal distribution. We then compute $b = A\bar{x}$ as the measurements, and apply each solver to produce a reconstruction vector $x^*$ of $\bar{x}$. A reconstruction is considered successful if the relative error satisfies $\|x^*-\bar{x}\|_2/\|\bar{x}\|_2< 10^{-3}$. We test some cases with different sparsity $s$ of $\bar{x}$, different levels of noise, or different numbers of measurements. We run 100 times independently for each scenario and report the success rate, which is the ratio of the number of successful trials over 100. All experiments are run in parallel with the MATLAB Parallel Computing Toolbox.

The initial guess for all tested algorithms is $x^0=0$. The choice of the parameter $\alpha$ in the springback penalty is discussed in Section \ref{sec:alpha}. For outer iterates of the DCA-springback, we set $\rho=10^{5}$, $\text{MaxIt} = 10$, and $\epsilon_{\text{outer}}=10^{-5}$ (for noise-free measurements) or $10^{-3}$ (for noisy measurements). To implement the ADMM \eqref{equ:ADMM} for subproblems, we set $\zeta=10^{-5}$, $\tau = \|A\bar{x}-b\|_2$, and the stopping criterion as either $\|x^{j+1}-x^j\|_2/\max\{\|x^{j+1}\|_2,\|x^j\|_2\}<10^{-5}$ or the iteration number exceeds $500$. The DCA-TL1, the DCA-$\ell_{1-2}$, and the DCA-MCP are solved by DCA and their subproblems are also solved by the ADMM. We thus set the regularization parameter $\lambda = 10^{-6}$ and adopt the same parameters of the rest and stopping criterion as the DCA-springback. In particular, the parameter $\beta$ in the transformed $\ell_1$ penalty is set as 1 for the DCA-TL1, following \cite{zhang2018minimization}, and the parameter $\mu$ in the MCP is set as $1/\alpha$ for the DCA-MCP. For the AIHT, we set all parameters as \cite{blumensath2009iterative}. For the ADMM-$\ell_1$, we set $\lambda = 10^{-6}$, $\zeta = 10^{-5}$, $\epsilon_{\text{outer}}=10^{-5}$ (for noise-free measurements) and $10^{-3}$ (for noisy measurements), and $\text{MaxIt} = 5000$. For IRLS-$\ell_p$, we set $p = 0.5$, $\lambda = 10^{-6}$, $\epsilon_{\text{outer}}=10^{-8}$, and $\text{MaxIt} = 1000$.

\subsection{A subroutine for choosing the model parameter $\alpha$}\label{sec:alpha}
Let us focus on the parameter $\alpha$ of the springback penalty \eqref{equ:spb}. For an $128\times 512$ random Gaussian matrix, we test the DCA-springback with different $\alpha$ varying among $\{0.2, 0.4, 0.6, 0.8, 1\}$, and different levels of sparsity $s$ among $\{25,27,\ldots,65\}$. The DCA-springback with $\alpha=0.6$ or $0.8$, indicated by success rates in Figure \ref{fig:alpha}, has the best performance. For small $\alpha$ such as 0.2 and 0.4, the DCA-springback is not satisfactory because the springback penalty performs similarly to the $\ell_1$ penalty. For $\alpha=1$, its performance is also inferior since the convergence condition of the DCA-springback or the posterior verification \eqref{equ:alpha} can be easily violated with a large $\alpha$. We refer to the latter reason as the ``violating behavior'' of the DCA-springback. An ``unsuccessful'' trial is recognized due to unsatisfactory (but reasonable) recovery or violating behavior. Thus, success rates cannot fully reflect ``violating behavior,'' and we also plot the relative errors in Figure \ref{fig:alpha}. Indeed, the ``violating behavior'' often occurs when $s$ becomes large. Performance of $\alpha=0.8$ and $1$ is generally inferior, and also there are few such cases when $\alpha=0.6$. Thus, we adopt a \texttt{safeguard} for $\alpha=0.7$, a compromise between $0.6$ and $0.8$. If $\alpha=0.7$ violates the condition \eqref{equ:alpharequirement}, then we replace 0.7 with the largest constant complying with this condition \eqref{equ:alpharequirement}. That is, we choose $\alpha=\min\left\{0.7,2\sigma_{\min}(A)/(\|b\|_2+\tau)\right\}$. Success rates and relative errors with \texttt{safeguarded} $\alpha=0.7$ are also displayed in Figure \ref{fig:alpha}, indicating that there is no violating behavior.
\begin{figure}[htbp]
  \centering
  \includegraphics[width=\textwidth]{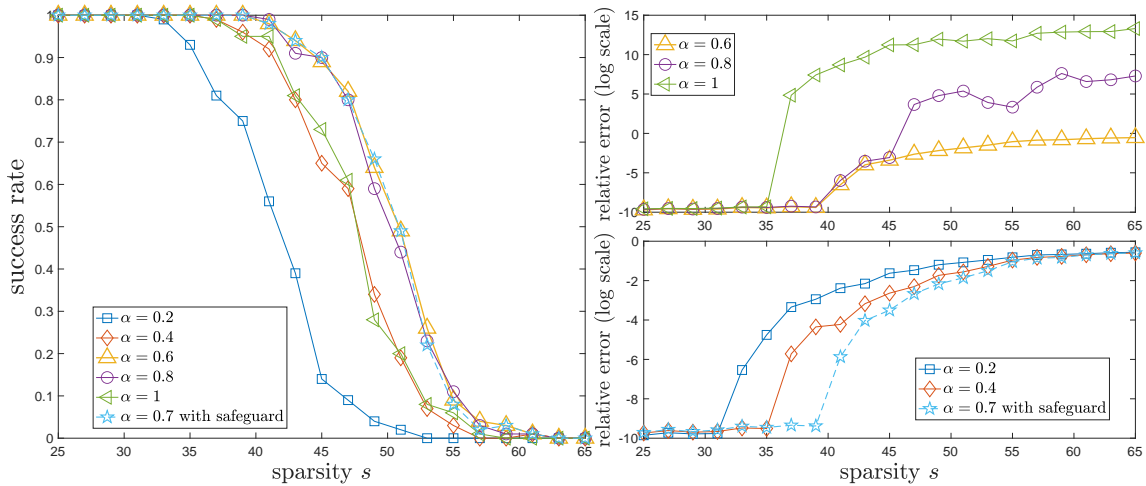}
  \caption{Success rates and relative errors in a natural logarithmic scale of recovery using DCA-springback under $128\times 512$ random Gaussian sensing matrix, with various $\alpha$.}\label{fig:alpha}
\end{figure}

Though a reasonable upper bound of $\alpha$ is needed, behaviors for $\alpha=0.2$ and $0.4$ suggest that a lower bound for $\alpha$ should be taken to maintain the satisfactory performance of the DCA-springback in terms of success rates. Especially if $A$ is ill-conditioned in the sense that its singular values lie within a wide range of values, i.e., $\sigma_{\min}(A)$ could be very small, then the condition on $\alpha$ could be pretty stringent. To maintain the success rates of the DCA-springback, we adopt an \texttt{efficiency detection} step as follows. If the condition number $\text{cond}(A):=\sigma_{\max}(A)/\sigma_{\min}(A)$ is greater than 5 (or other values set by the user), then we start an \texttt{efficiency detection} to enforce $\alpha$ to be greater than an \emph{efficiency detection factor} $\omega$. Thus, we suggest choosing $\alpha$ as the following subroutine:
\begin{equation}\label{equ:subroutine}
\alpha = \begin{cases}
\min\left\{0.7,2\sigma_{\min}(A)/(\|b\|_2+\tau)\right\}, & \text{if }\text{cond}(A)\leq 5,\\
\max\left\{\omega,\min\{0.7,2\sigma_{\min}(A)/(\|b\|_2+\tau)\}\right\},& \text{otherwise}.
\end{cases}\end{equation}
In short, the \texttt{safeguard} step suffices to guarantee convergence of the DCA-springback; and the \texttt{efficiency detection} step is adopted to maintain the success rates of the DCA-springback for ill-conditioned sensing matrices.

\subsection{Exact recovery of sparse vectors}\label{sec:noisefreeexperiment}
We first compare the DCA-springback with some state-of-the-art solvers mentioned above for noise-free measurements. We consider both the incoherent and coherent sensing matrices, respectively.

\textbf{Tests on incoherent matrices.} We first consider a ground-truth vector and display its reconstructions by the ADMM-$\ell_1$, the DCA-TL1, the DCA-${\ell_{1-2}}$, the DCA-MCP, and the DCA-springback. Let the sensing matrix $A\in\mathbb{R}^{m\times n}$ be a random Gaussian matrix with $(m,n)=(64,250)$, and the ground-truth $\bar{x}\in\mathbb{R}^{250}$ be a $22$-sparse vector with nonzero entries drawn from the standard normal distribution and set the efficiency detection factor as $\omega= 0.5$. The ground-truth and its reconstructions are displayed in Figure \ref{fig:groundtruth}. We see that the DCA-springback, the DCA-MCP, and the DCA-TL1 produce better reconstructions than the ADMM-$\ell_1$ and the DCA-$\ell_{1-2}$.

\begin{figure}[htbp]
  \centering
  \includegraphics[width=\textwidth]{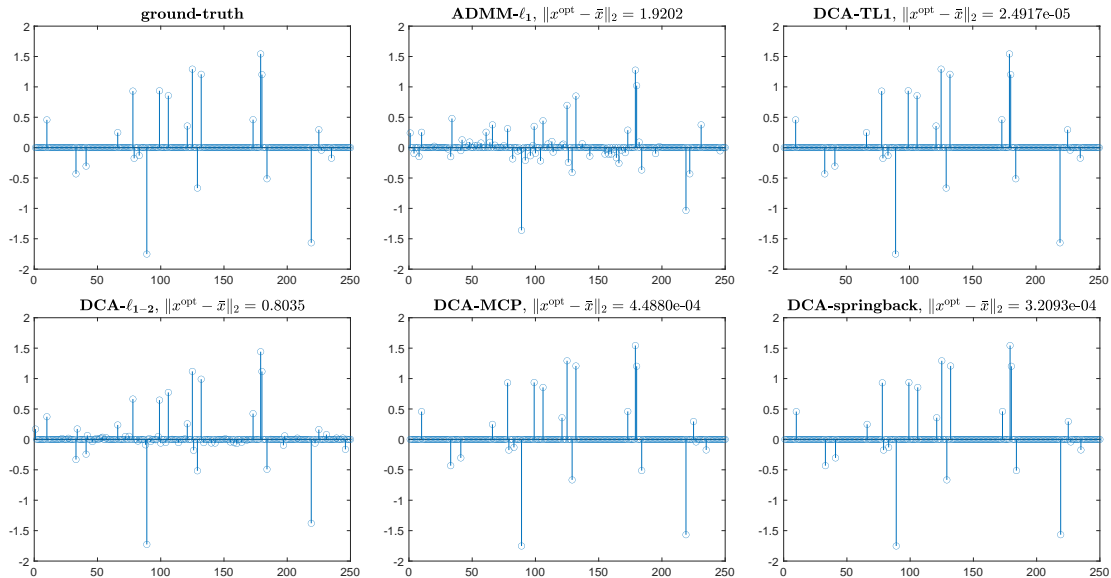}
  \caption{A ground-truth and its reconstructions using random Gaussian sensing matrices and noise-free measurements.}\label{fig:groundtruth}
\end{figure}

We then conduct a more comprehensive study and involve more solvers. We choose the sensing matrix $A\in\mathbb{R}^{m\times n}$ as a random Gaussian matrix and random partial DCT matrices with $(m,n)=(64,160)$, $(64,320)$, and $(64,640)$, and set the efficiency detection factor as $\omega= 0.5$. Different levels of sparsity $s$ varying among $\{6,8,10,\ldots,40\}$ are tested. The success rates of each solver are plotted in Figure \ref{fig:incoherent}. For both the Gaussian and partial DCT matrices, the IRLS-$\ell_p$ with $p = 0.5$ has the best performance, followed by the DCA-TL1, the DCA-MCP, and the DCA-springback. In particular, the performances of the DCA-MCP and the DCA-springback are very close because we let the parameter $\mu$ in the MCP be $1/\alpha$. The DCA-$\ell_{1-2}$ performs moderately well,  outperforming both the ADMM-$\ell_1$ and the AIHT. Our numerical results are consistent with some observations in the literature (e.g., \cite{yin2015minimization,zhang2018minimization}).
\begin{figure}[htbp]
  \centering
  \includegraphics[width=\textwidth]{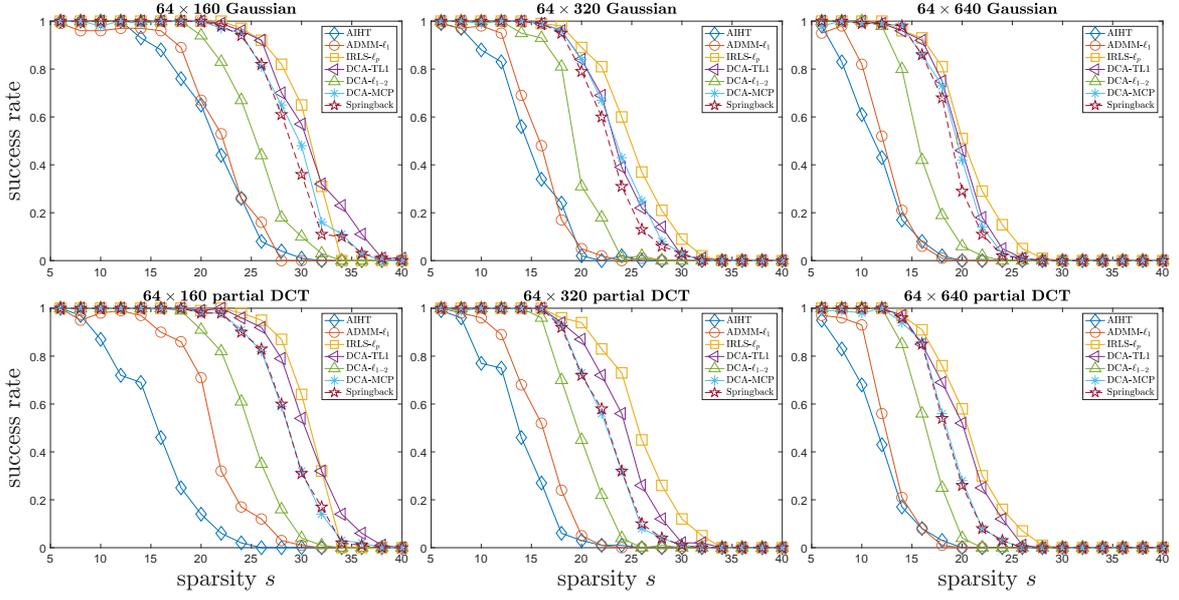}
  \caption{Success rates using random Gaussian and partial DCT sensing matrices.}\label{fig:incoherent}
\end{figure}

\textbf{Tests on coherent matrices.} Now, we choose the sensing matrix $A\in\mathbb{R}^{100\times 1500}$ as a randomly oversampled partial DCT matrix with various refinement factors $\mathcal{F}=4,6,8,10,12,16$ and minimum separation $L=2\mathcal{F}$, with the sparsity $s$ varying among $\{5,7,9,\ldots,35\}$. The efficiency detection factor is set as $\omega=0.5$. The success rates of each solver are plotted in Figure \ref{fig:coherent}. This figure  suggests that the DCA-TL1, the DCA-MCP, and the DCA-springback are robust regardless of the varying coherence of sensing matrix $A$. Moreover, when the coherent of $A$ is modest, e.g. $\mathcal{F}=6,8$, the DCA-MCP and the DCA-springback perform better than others. In the coherent regime, the DCA-springback is comparable with the DCA-$\ell_{1-2}$, and it outperforms the DCA-TL1, the ADMM-$\ell_1$, the IRLS-$\ell_p$, and the AIHT. However, the best-performance solver IRLS-$\ell_p$ in the incoherent regime becomes inefficient as $A$ becomes coherent.

\begin{figure}[htbp]
  \centering{}
  \includegraphics[width=\textwidth]{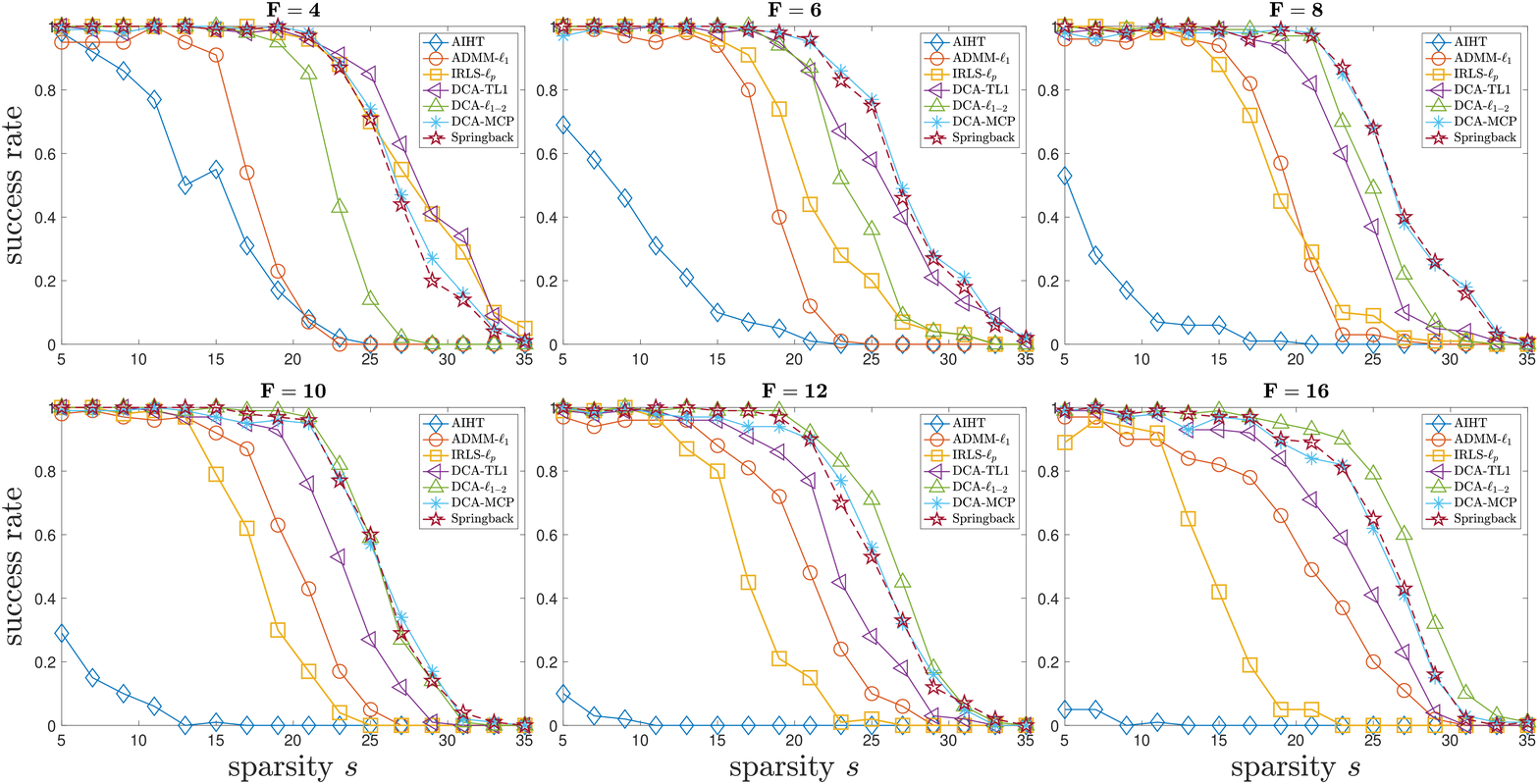}
  \caption{Success rates using randomly oversampled partial DCT matrices in $\mathbb{R}^{100\times1500}$.}\label{fig:coherent}
\end{figure}

\subsection{Robust recovery in the presence of noise}\label{sec:noisyexperiment}
We then consider noisy measurements. The noisy measurements $b$ are obtained by \texttt{b = awgn(A$\bar{\rm x}$,snr)}, a subroutine of the MATLAB Communication Toolbox, where \texttt{snr} corresponds to the value of signal-to-noise ratio (SNR) measured in dB. The larger the value of SNR is, the lighter the noise is added on.

We first consider a ground-truth vector with noisy measurements and display its reconstructions by the ADMM-$\ell_1$, the DCA-TL1, the DCA-${\ell_{1-2}}$, the DCA-MCP, and the DCA-springback. Let the sensing matrix $A\in\mathbb{R}^{m\times n}$ be a random Gaussian matrix with $(m,n)=(64,250)$, and the ground-truth $\bar{x}\in\mathbb{R}^{250}$ be a $20$-sparse vector with nonzero entries drawn from the standard normal distribution and set the efficiency detection factor as $\omega= 0.4$. The measurement vector $b= A\bar{x}$ is contaminated by 30 dB noise. The ground-truth and its reconstructions are displayed in Figure \ref{fig:groundtruthnoisy}. In particular, we see that the DCA-springback works better on small perturbations than the other solvers.

\begin{figure}[htbp]
  \centering
  \includegraphics[width=\textwidth]{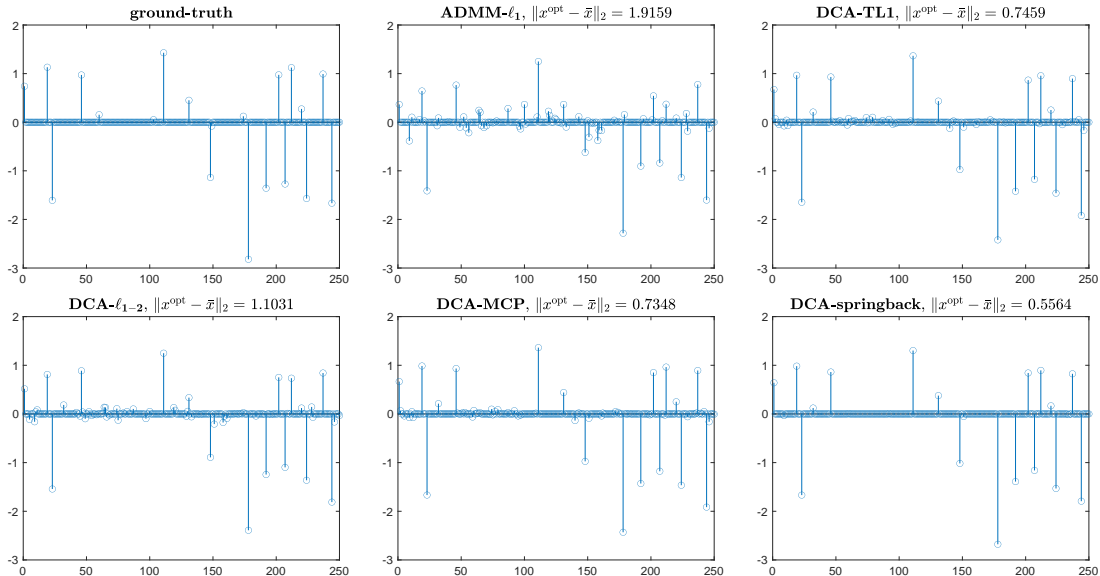}
  \caption{A ground-truth and its reconstructions using random Gaussian sensing matrices and noisy measurements.}\label{fig:groundtruthnoisy}
\end{figure}

We test both the random Gaussian matrix and the randomly oversampled partial DCT matrix with different levels of noise in dB. For Gaussian measurements, we choose $n = 64$, $m = 128$, and $s = 25$. For the oversampled partial DCT measurements, we test $n = 1500$, $m = 128$, $s = 30$, and $\mathcal{F} = 8$. We run 100 times for each scenario and record the average errors. The efficiency detection factor is set as $\omega=0.4$.

Once we adopt the \texttt{efficiency detection} step, a single ``violating behavior'' could lift the mean error to a pretty large level. To overcome this computational myopia, we only reserve the accepted results, where a result of the DCA-springback is considered ``accepted'' if the absolute error $\|x^*-\bar{x}\|_2$ is ten times less than the absolute error of the ADMM-$\ell_1$. In addition to errors displayed in Figure \ref{fig:noisy}, we report the \emph{acceptance rates} of the DCA-springback, which are ratios of the number of accepted trials over 100.

According to our experiments, there are no ``violating behaviors'' with the Gaussian measurements. However, there are a few cases with the oversampled partial DCT measurements when the noise level is relatively large. To illustrate the necessity of the \texttt{efficiency detection} step and to validate the convergence condition \eqref{equ:alpharequirement}, we test the DCA-springback without the \texttt{efficiency detection} for the randomly oversampled partial DCT measurements, and we do not remove unaccepted trials. The results are labeled as ``DCA-springback w/o effcy det.'' in Figure \ref{fig:noisy}, as we see that the DCA-springback only performs slightly better than the ADMM-$\ell_1$.

\begin{figure}[htbp]
  \centering
  \includegraphics[width=\textwidth]{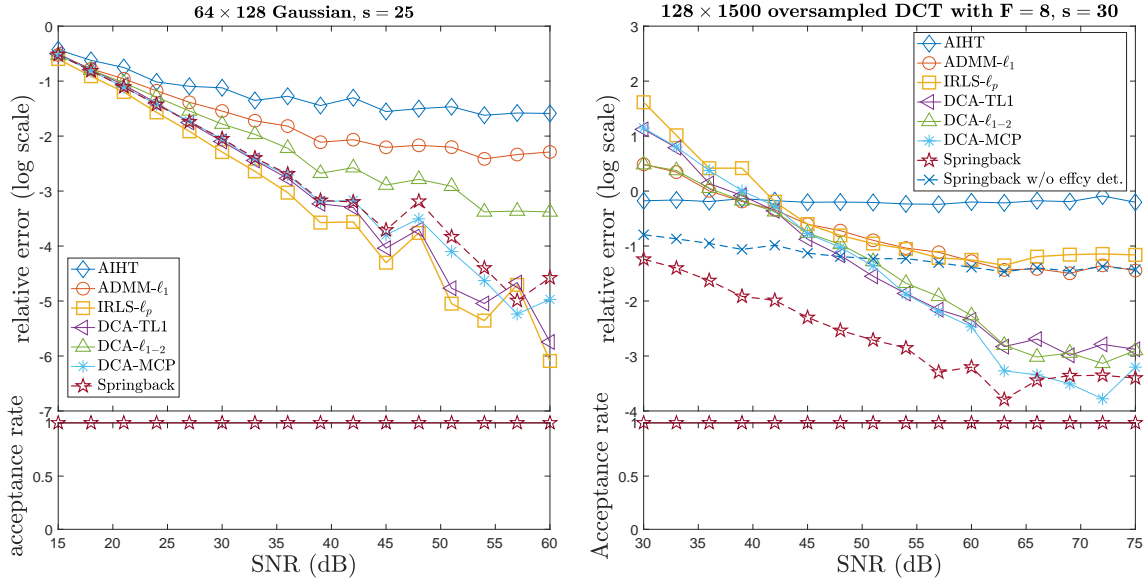}
  \caption{Robust recovery results with randomly Gaussian and oversampled partial DCT measurements.}\label{fig:noisy}
\end{figure}

Figure \ref{fig:noisy} shows that the DCA-$\ell_{1-2}$ and the IRLS-$\ell_p$ are still sensitive to the coherence of $A$. For Gaussian measurements, the IRLS-$\ell_p$ with $p = 0.5$ has the best performance, followed by the DCA-TL1, the DCA-MCP, the DCA-springback, the DCA-$\ell_{1-2}$, and the ADMM-$\ell_1$. For oversampled DCT measurements, the DCA-springback appears to be the best solver, followed by the DCA-MCP, the DCA-$\ell_{1-2}$, and the DCA-TL1, because the noise level is considered in solving the subproblems of the DCA-springback. In both cases, the DCA-springback consistently performs better than the ADMM-$\ell_1$ and the DCA-$\ell_{1-2}$. AIHT appears not to perform well for both matrices. According to the plots of the DCA-springback and the DCA-springback without the \texttt{efficiency detection}, the model parameter $\alpha$ matters for the same solver.

We also validate some theoretical results proved in Section \ref{sec:stable3}, with Gaussian measurements perturbed by 45 dB noise. We first study $m=50$, $n=160$, and $s$ varying among $\{10,11,\ldots,40\}$, and then consider $n=160$, $s=20$, and $m$ varying among $\{50,51,\ldots,120\}$. Errors of the ADMM-$\ell_1$, the DCA-$\ell_{1-2}$, and the DCA-springback are plotted in Figure \ref{fig:noisy2}, and the acceptance rates of the DCA-springback are also displayed. According to our analysis in Section \ref{sec:stable3}, for an RIP sensing matrix $A$ and an $s$-sparse $\bar{x}$, when $s\leq C$ ($C$ is given in \eqref{equ:C}) or $m$ is limited by some constant, the estimation \eqref{equ:stablerecoverynonsparse} of the springback-penalized model is tighter than the estimation \eqref{equ:l1nonsparserecovery} of the $\ell_1$- and $\ell_{1-2}$-penalized models in the sense of \eqref{equ:comp}. We see in the left plot of Figure \ref{fig:noisy2} that the error of the DCA-springback is less than the others for small $s$, and it becomes larger than the others when $s$ exceeds some constant. The right plot also indicates that the error of the DCA-springback is less than the others when $m$ is relatively small.
\begin{figure}[htbp]
  \centering
  \includegraphics[width=\textwidth]{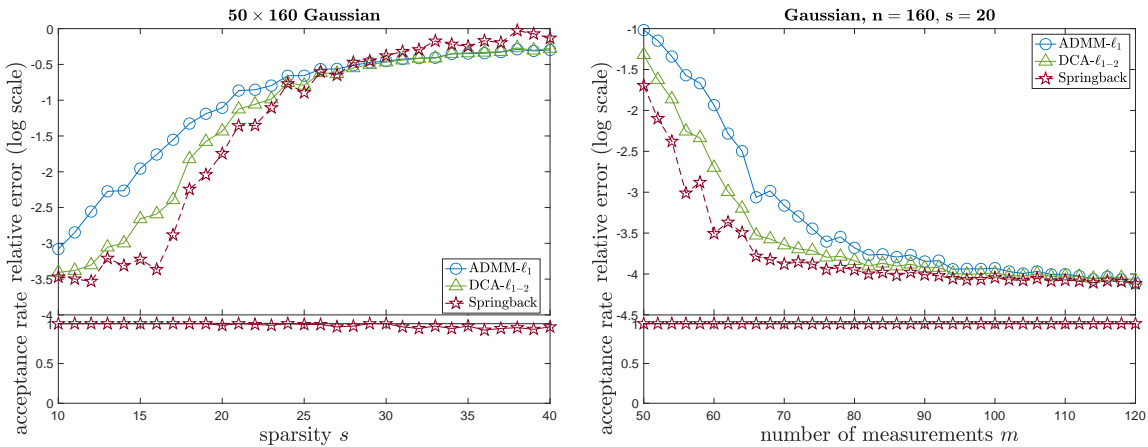}
  \caption{Numerical validation of theoretical results in Section \ref{sec:stable3}.}\label{fig:noisy2}
\end{figure}

\subsection{Remarks on numerical results}
As observed in the literature, recovery results by different models may vary for different scenarios, and no one can unanimously outperform all the others for all scenarios. For instance, the IRLS-$\ell_p$ prevails in the incoherent regime but quickly fails in the coherent regime, see \cite{lai2013improved,yin2015minimization}. For incoherent sensing matrices, the IRLS-$\ell_p$ and the DCA-TL1 perform better than the DCA-$\ell_{1-2}$ and the ADMM-$\ell_1$, while the DCA-$\ell_{1-2}$ performs the best for coherent sensing matrices; see \cite{yin2015minimization,zhang2018minimization}. The DCA-TL1 is robust, and it performs well for both incoherent and coherent sensing matrices, while it is less efficient than either the IRLS-$\ell_p$ in the incoherent regime or the DCA-$\ell_{1-2}$ in the coherent regime.

Together with these known facts and our numerical observations, we have the following remarks on the numerical performance of the DCA-springback.
\begin{itemize}
       \item For an incoherent sensing matrix: the DCA-springback performs slightly worse than the IRLS-$\ell_p$ and the DCA-TL1;
       \item For a coherent sensing matrix: the DCA-springback performs slightly worse than the DCA-$\ell_{1-2}$ but better than the DCA-TL1.
       \item For a sensing matrix with modest coherence: the DCA-springback performs comparably with the DCA-MCP, and they perform better than the others.
\end{itemize}
Similar comparison results are also observed when the measurements are contaminated by some noise.
For all the three scenarios, the DCA-springback and the DCA-MCP perform comparably if the parameter $\mu$ of the MCP is set as $1/\alpha$, and their performances with well-tuned parameters are also comparable. Moreover, we see that only the DCA-springback, the DCA-MCP, and the DCA-TL1 are robust with respect to the coherence of the sensing matrix. The DCA-springback and the DCA-MCP perform better than the DCA-TL1 in the coherent regime but worse in the incoherent regime. When the coherence of the sensing matrix is unknown, for example, when the sensing hardware cannot be modified or upgraded, coherence-robust algorithms such as the DCA-springback and the DCA-MCP are preferred for signal recovery.

\section{Conclusion}\label{sec:conclusions}

We proposed a weakly convex penalty, named the springback penalty, for signal recovery from incomplete and inaccurate measurements. The springback penalty inherits major theoretical and numerical advantages from the convex $\ell_1$ penalty and its various non-convex alternatives. We established exact and stable recovery results for the springback-penalized model \eqref{equ:CSproblem} under the same RIP condition as the BP model \eqref{equ:l1}; both the sparse and nearly sparse signals are considered. The springback-penalized model \eqref{equ:CSproblem} is particularly suitable for signal recovery with a large level of noise or a limited number of measurements. We verified the effectiveness of the model and its computational tractability. The springback penalty provides a new tool to construct effective models for various sparsity-driven recovery problems arising in many areas such as compressed sensing, signal processing, image processing, and least-squares approximation.

\section*{Acknowledgement}
The authors are grateful to the anonymous referees for their very valuable comments which have helped them improve this work substantially.

\bibliographystyle{siam}
\bibliography{myref}

\begin{thebibliography}{10}

\bibitem{MR3719240}
{\sc A.~Beck}, {\em First-Order Methods in Optimization}, SIAM, Philadelphia;
  Mathematical Optimization Society, Philadelphia, 2017.

\bibitem{blumensath2009iterative}
{\sc T.~Blumensath and M.~E. Davies}, {\em Iterative hard thresholding for
  compressed sensing}, Applied and Computational Harmonic Analysis, 27 (2009),
  pp.~265--274.

\bibitem{bruckstein2009sparse}
{\sc A.~M. Bruckstein, D.~L. Donoho, and M.~Elad}, {\em From sparse solutions
  of systems of equations to sparse modeling of signals and images}, SIAM
  Review, 51 (2009), pp.~34--81.

\bibitem{candes2006robust}
{\sc E.~J. Cand{\`e}s, J.~Romberg, and T.~Tao}, {\em Robust uncertainty
  principles: Exact signal reconstruction from highly incomplete frequency
  information}, IEEE Transactions on Information Theory, 52 (2006),
  pp.~489--509.

\bibitem{candes2006stable}
{\sc E.~J. Cand{\`e}s, J.~K. Romberg, and T.~Tao}, {\em Stable signal recovery
  from incomplete and inaccurate measurements}, Communications on Pure and
  Applied Mathematics, 59 (2006), pp.~1207--1223.

\bibitem{candes2005error}
{\sc E.~J. Cand{\`e}s, M.~Rudelson, T.~Tao, and R.~Vershynin}, {\em Error
  correction via linear programming}, in 46th Annual IEEE Symposium on
  Foundations of Computer Science (FOCS'05), IEEE, 2005, pp.~668--681.

\bibitem{candes2005decoding}
{\sc E.~J. Cand{\`e}s and T.~Tao}, {\em Decoding by linear programming}, IEEE
  Transactions on Information Theory, 51 (2005), pp.~4203--4215.

\bibitem{candes2006near}
\leavevmode\vrule height 2pt depth -1.6pt width 23pt, {\em Near-optimal signal
  recovery from random projections: {U}niversal encoding strategies?}, IEEE
  Transactions on Information Theory, 52 (2006), pp.~5406--5425.

\bibitem{chartrand2007exact}
{\sc R.~Chartrand}, {\em Exact reconstruction of sparse signals via nonconvex
  minimization}, IEEE Signal Processing Letters, 14 (2007), pp.~707--710.

\bibitem{chartrand2008restricted}
{\sc R.~Chartrand and V.~Staneva}, {\em Restricted isometry properties and
  nonconvex compressive sensing}, Inverse Problems, 24 (2008), p.~035020.

\bibitem{chen2016direct}
{\sc C.~Chen, B.~He, Y.~Ye, and X.~Yuan}, {\em The direct extension of {ADMM}
  for multi-block convex minimization problems is not necessarily convergent},
  Mathematical Programming, 155 (2016), pp.~57--79.

\bibitem{chen2001atomic}
{\sc S.~S. Chen, D.~L. Donoho, and M.~A. Saunders}, {\em Atomic decomposition
  by basis pursuit}, SIAM Review, 43 (2001), pp.~129--159.

\bibitem{donoho2006compressed}
{\sc D.~L. Donoho}, {\em Compressed sensing}, IEEE Transactions on Information
  Theory, 52 (2006), pp.~1289--1306.

\bibitem{esser2013method}
{\sc E.~Esser, Y.~Lou, and J.~Xin}, {\em A method for finding structured sparse
  solutions to nonnegative least squares problems with applications}, SIAM
  Journal on Imaging Sciences, 6 (2013), pp.~2010--2046.

\bibitem{fan2001variable}
{\sc J.~Fan and R.~Li}, {\em Variable selection via nonconcave penalized
  likelihood and its oracle properties}, Journal of the American Statistical
  Association, 96 (2001), pp.~1348--1360.

\bibitem{foucart2011hard}
{\sc S.~Foucart}, {\em Hard thresholding pursuit: an algorithm for compressive
  sensing}, SIAM Journal on Numerical Analysis, 49 (2011), pp.~2543--2563.

\bibitem{foucart2009sparsest}
{\sc S.~Foucart and M.-J. Lai}, {\em Sparsest solutions of underdetermined
  linear systems via $\ell_q$-minimization for $0<q\leq 1$}, Applied and
  Computational Harmonic Analysis, 26 (2009), pp.~395--407.

\bibitem{MR3100033}
{\sc S.~Foucart and H.~Rauhut}, {\em A Mathematical Introduction to Compressive
  Sensing}, Applied and Numerical Harmonic Analysis, Birkh\"{a}user, Basel,
  2013.

\bibitem{gao1997waveshrink}
{\sc H.-Y. Gao and A.~G. Bruce}, {\em Wave{S}hrink with firm shrinkage},
  Statistica Sinica,  (1997), pp.~855--874.

\bibitem{MR388811}
{\sc R.~Glowinski and A.~Marrocco}, {\em Sur l'approximation, par
  \'{e}l\'{e}ments finis d'ordre un, et la r\'{e}solution, par
  p\'{e}nalisation-dualit\'{e}, d'une classe de probl\`emes de {D}irichlet non
  lin\'{e}aires}, Revue Fran\c{c}aise d'Automatique, Informatique et Recherche
  Op\'{e}rationnelle S\'{e}rie Rouge. Analyse Num\'{e}rique, 9 (1975),
  pp.~41--76.

\bibitem{guo2017convergence}
{\sc K.~Guo, D.~Han, and X.~Yuan}, {\em Convergence analysis of
  {D}ouglas--{R}achford splitting method for ``strongly+weakly'' convex
  programming}, SIAM Journal on Numerical Analysis, 55 (2017), pp.~1549--1577.

\bibitem{he20121}
{\sc B.~He and X.~Yuan}, {\em On the ${O}(1/n)$ convergence rate of the
  {D}ouglas--{R}achford alternating direction method}, SIAM Journal on
  Numerical Analysis, 50 (2012), pp.~700--709.

\bibitem{huang2018constructive}
{\sc J.~Huang, Y.~Jiao, Y.~Liu, and X.~Lu}, {\em A constructive approach to
  ${L}_0$ penalized regression}, The Journal of Machine Learning Research, 19
  (2018), pp.~403--439.

\bibitem{MR4201341}
{\sc M.~Huang, M.-J. Lai, A.~Varghese, and Z.~Xu}, {\em On {DC} based methods
  for phase retrieval}, in Approximation theory {XVI}, Springer, Cham, 2021,
  pp.~87--121.

\bibitem{MR3398943}
{\sc Y.~Jiao, B.~Jin, and X.~Lu}, {\em A primal dual active set with
  continuation algorithm for the {$\ell^0$}-regularized optimization problem},
  Applied and Computational Harmonic Analysis, 39 (2015), pp.~400--426.

\bibitem{jiao2017iterative}
\leavevmode\vrule height 2pt depth -1.6pt width 23pt, {\em Iterative soft/hard
  thresholding with homotopy continuation for sparse recovery}, IEEE Signal
  Processing Letters, 24 (2017), pp.~784--788.

\bibitem{MR2765490}
{\sc M.-J. Lai and J.~Wang}, {\em An unconstrained {$\ell_q$} minimization with
  {$0<q\leq 1$} for sparse solution of underdetermined linear systems}, SIAM
  Journal on Optimization, 21 (2011), pp.~82--101.

\bibitem{lai2013improved}
{\sc M.-J. Lai, Y.~Xu, and W.~Yin}, {\em Improved iteratively reweighted least
  squares for unconstrained smoothed $\ell_q$ minimization}, SIAM Journal on
  Numerical Analysis, 51 (2013), pp.~927--957.

\bibitem{lv2009unified}
{\sc J.~Lv and Y.~Fan}, {\em A unified approach to model selection and sparse
  recovery using regularized least squares}, The Annals of Statistics, 37
  (2009), pp.~3498--3528.

\bibitem{mendelson2007reconstruction}
{\sc S.~Mendelson, A.~Pajor, and N.~Tomczak-Jaegermann}, {\em Reconstruction
  and subgaussian operators in asymptotic geometric analysis}, Geometric and
  Functional Analysis, 17 (2007), pp.~1248--1282.

\bibitem{mollenhoff2015primal}
{\sc T.~M\"{o}llenhoff, E.~Strekalovskiy, M.~Moeller, and D.~Cremers}, {\em The
  primal-dual hybrid gradient method for semiconvex splittings}, SIAM Journal
  on Imaging Sciences, 8 (2015), pp.~827--857.

\bibitem{MR201952}
{\sc J.-J. Moreau}, {\em Proximit\'{e} et dualit\'{e} dans un espace
  hilbertien}, Bulletin de la Soci\'{e}t\'{e} Math\'{e}matique de France, 93
  (1965), pp.~273--299.

\bibitem{needell2009cosamp}
{\sc D.~Needell and J.~A. Tropp}, {\em {CoSaMP}: {I}terative signal recovery
  from incomplete and inaccurate samples}, Applied and Computational Harmonic
  Analysis, 26 (2009), pp.~301--321.

\bibitem{pati1993orthogonal}
{\sc Y.~C. Pati, R.~Rezaiifar, and P.~S. Krishnaprasad}, {\em Orthogonal
  matching pursuit: {R}ecursive function approximation with applications to
  wavelet decomposition}, in Proceedings of 27th Asilomar Conference on
  Signals, Systems and Computers, IEEE, 1993, pp.~40--44.

\bibitem{MR0274683}
{\sc R.~T. Rockafellar}, {\em Convex analysis}, Princeton Mathematical Series,
  No. 28, Princeton University Press, Princeton, 1970.

\bibitem{rudelson2008sparse}
{\sc M.~Rudelson and R.~Vershynin}, {\em On sparse reconstruction from
  {F}ourier and {G}aussian measurements}, Communications on Pure and Applied
  Mathematics, 61 (2008), pp.~1025--1045.

\bibitem{saab2008stable}
{\sc R.~Saab, R.~Chartrand, and O.~Yilmaz}, {\em Stable sparse approximations
  via nonconvex optimization}, in 2008 IEEE International Conference on
  Acoustics, Speech and Signal Processing, IEEE, 2008, pp.~3885--3888.

\bibitem{sun2018sparse}
{\sc Y.~Sun, H.~Chen, and J.~Tao}, {\em Sparse signal recovery via
  minimax-concave penalty and $\ell_1-$norm loss function}, IET Signal
  Processing, 12 (2018), pp.~1091--1098.

\bibitem{tao1997convex}
{\sc P.~D. Tao and L.~T.~H. An}, {\em Convex analysis approach to {DC}
  programming: theory, algorithms and applications}, Acta Mathematica
  Vietnamica, 22 (1997), pp.~289--355.

\bibitem{tao1998dc}
\leavevmode\vrule height 2pt depth -1.6pt width 23pt, {\em A {DC} optimization
  algorithm for solving the trust-region subproblem}, SIAM Journal on
  Optimization, 8 (1998), pp.~476--505.

\bibitem{todd1994manufacturing}
{\sc R.~H. Todd, D.~K. Allen, and L.~Alting}, {\em Manufacturing Processes
  Reference Guide}, Industrial Press, Inc., New York, 1994.

\bibitem{MR1444820}
{\sc L.~N. Trefethen and D.~Bau, III}, {\em Numerical Linear Algebra}, SIAM,
  Philadelphia, 1997.

\bibitem{woodworth2016compressed}
{\sc J.~Woodworth and R.~Chartrand}, {\em Compressed sensing recovery via
  nonconvex shrinkage penalties}, Inverse Problems, 32 (2016), p.~075004.

\bibitem{yan2017sparse}
{\sc L.~Yan, Y.~Shin, and D.~Xiu}, {\em Sparse approximation using
  $\ell_1-\ell_2$ minimization and its application to stochastic collocation},
  SIAM Journal on Scientific Computing, 39 (2017), pp.~A229--A254.

\bibitem{yin2015minimization}
{\sc P.~Yin, Y.~Lou, Q.~He, and J.~Xin}, {\em Minimization of $\ell_{1-2}$ for
  compressed sensing}, SIAM Journal on Scientific Computing, 37 (2015),
  pp.~A536--A563.

\bibitem{zhang2010nearly}
{\sc C.-H. Zhang}, {\em Nearly unbiased variable selection under minimax
  concave penalty}, The Annals of Statistics, 38 (2010), pp.~894--942.

\bibitem{MR3620567}
{\sc S.~Zhang and J.~Xin}, {\em Minimization of transformed {$L_1$} penalty:
  closed form representation and iterative thresholding algorithms},
  Communications in Mathematical Sciences, 15 (2017), pp.~511--537.

\bibitem{zhang2018minimization}
\leavevmode\vrule height 2pt depth -1.6pt width 23pt, {\em Minimization of
  transformed ${L}_1$ penalty: theory, difference of convex function algorithm,
  and robust application in compressed sensing}, Mathematical Programming, 169
  (2018), pp.~307--336.

\bibitem{MR2629825}
{\sc T.~Zhang}, {\em Analysis of multi-stage convex relaxation for sparse
  regularization}, Journal of Machine Learning Research, 11 (2010),
  pp.~1081--1107.

\bibitem{MR2137327}
{\sc H.~Zou and T.~Hastie}, {\em Regularization and variable selection via the
  elastic net}, Journal of the Royal Statistical Society. Series B. Statistical
  Methodology, 67 (2005), pp.~301--320.

\end{thebibliography}
\clearpage

\end{document}